\documentclass[11pt]{article}
\usepackage{mathrsfs}
\usepackage{makeidx,ifthen}
\usepackage{latexsym}
\usepackage{graphicx}
\usepackage{amsmath}
\usepackage{amssymb}
\usepackage{amsbsy,amsthm}
\usepackage{amsfonts}
\usepackage{wasysym}
\usepackage{enumerate}
\usepackage[svgnames,x11names,dvipsnames]{xcolor}

\newtheorem{thm}{Theorem}[section]
\newtheorem{cor}[thm]{Corollary}
\newtheorem{lem}[thm]{Lemma}
\newtheorem{prop}[thm]{Proposition}
\newtheorem{rem}[thm]{Remark}
\numberwithin{equation}{section}

\begin{document}

\title{Normalized solutions to 1-Laplace equations with critical growth term \footnotemark[2] \footnotetext[2]
{Supported by the Natural Science Foundation of China (No.12571191, No.12501560) and the Hunan Basic Science Research Center for Mathematical Analysis (2024JC2002). }}
\author{Zhouxin Li$^a$\quad Ruishu Liu$^{b,c}$\thanks{Corresponding author}\\
\small\it $^a$School of Mathematics and Statistics, HNP-LAMA, Central South University,\\
\small\it Changsha, {\rm 410083}, China\\
\small\it E-mail: lzx@math.pku.edu.cn\\
\small\it $^b$School of Artificial Intelligence, Shenzhen University,\\
\small\it Shenzhen, {\rm 518060}, China\\
\small\it $^c$National Engineering Laboratory for Big Data System Computing Technology, Shenzhen University,\\
\small\it Shenzhen, {\rm 518060}, China\\
\small\it E-mail: rsliu@szu.edu.cn
}

\date{}
\maketitle \noindent{\small{\bf Abstract}\quad
We study normalized solutions to the 1-Laplace equation
\[
-\mbox{div}\left(\frac{Du}{|Du|}\right)+\lambda \frac{u}{|u|}=f(u)+|u|^{1^*-2}u,\quad\mbox{in}~ \mathbb{R}^N
\]
subject to the constraint $\int_{\mathbb{R}^N}|u| \, \mbox{d}x=m$, where $N\geq2$, $1^*=\frac{N}{N-1}$ is the critical Sobolev exponent for the embedding from BV into $L^p$, $\lambda, m\in\mathbb{R}^+$ and $f:\mathbb{R}\to\mathbb{R}$ is a continuous nonlinear perturbation.
We first establish a Pohozaev-type identity for the equation and use it to study the associated critical energy, and then prove the existence of normalized solutions via nonsmooth critical point theory.
A key feature of our approach is the proof of the existence of a Palais-Smale sequence for a Lipschitz continuous functional under a Lipschitz continuous constraint.
This is significantly different in methodology from previous problems with $C^2$ constraint.


\noindent{\small{\bf  MR(2020) Subject Classification}\quad 35J20; 35J75; 49J52}
\vspace{0.5cm}

\section{Introduction}
\indent\indent
In this paper, we study normalized solutions to a class of 1-Laplace equations with a critical growth term. 
More precisely, we consider the existence of a pair $(\lambda,u) \in \mathbb{R}\times BV(\mathbb{R}^N)$ satisfying
\begin{eqnarray}\label{eq:1-Lap}
    -\mbox{div}\left(\frac{Du}{|Du|}\right)+\lambda \frac{u}{|u|}=f(u)+|u|^{1^*-2}u, & \hbox{in $\mathbb{R}^N$}
\end{eqnarray}
subject to the constraint
\begin{eqnarray}\label{eq:constraint}
    \int_{\mathbb{R}^N}|u| \, \mbox{d}x=m, & \hbox{}
\end{eqnarray}
where $1^*=\frac{N}{N-1}$, $N\geq2$, $\lambda, m\in\mathbb{R}^+$, and $f:\mathbb{R}\to\mathbb{R}$ is a nonlinear perturbation;
here $BV(\mathbb{R}^N)$ denotes the space of functions of bounded variation on $\mathbb{R}^N$ (see \eqref{def:TV} below).

The motivation for studying this problem comes from the Schr\"{o}dinger equation in physics.
It is well known that the quasilinear Schr\"{o}dinger equation
\begin{eqnarray*}
i\frac{\partial\phi}{\partial t}=-\Delta\phi+W(x)\phi-g(|\phi|^2)\phi, & \hbox{in $\mathbb{R}\times\mathbb{R}^N$}
\end{eqnarray*}
has been widely studied by many authors. 
In particular, the study of standing-wave solutions, i.e. $\phi=e^{-iE t}u(x)$, transforms this equation into a quasilinear elliptic equation
\begin{eqnarray*}
-\Delta u+V(x) u-f(u)=0, & \hbox{in $\mathbb{R}^N$},
\end{eqnarray*}
where the potential $V(x)=W(x)+E$ and $f(u)=g(|u|^2)u$.
One can refer to \cite{ABC97,FlWe86,Rabi92} for semiclassical solutions, to \cite{PiFe96,Wang93} for concentration of solutions and to \cite{Gui96,Oh90} for multi-bump solutions.

In recent years, the study of the existence of solutions to the Schr\"{o}dinger equation
\begin{eqnarray*}\label{eq.sch1}
-\Delta u+\lambda u=f(u),\quad\mbox{in~}\mathbb{R}^N
\end{eqnarray*}
satisfying
\[
\int_{\mathbb{R}^N}u^2 \, \mbox{d}x=m^2,\quad u\in H^1(\mathbb{R}^N)
\]
which are called normalized solutions, has attracted considerable attention and has been intensively studied. 
In this equation, the parameter $\lambda$ can be regarded not only as a constant potential but also as a Lagrange multiplier under the so-called mass constraint.

One significant and early result was obtained by Jeanjean in \cite{Jean97}. 
He considered the semilinear elliptic equation
\begin{eqnarray*}
-\Delta u+\lambda u=g(u),\quad\mbox{in~}\mathbb{R}^N
\end{eqnarray*}
for $N\geq1$, $\lambda\in\mathbb{R}$ under the assumptions that $g$ satisfies

($H_1$) $g\in C(\mathbb{R},\mathbb{R})$ and $g$ is odd;

($H_2$) there exist $\alpha, \beta\in\mathbb{R}$ with $2+\frac{4}{N}<\alpha\leq\beta<2^*$ such that
\[
0<\alpha G(t)\leq g(t)t\leq\beta G(t)
\]
for all $t\in\mathbb{R}\setminus\{0\}$, where $G(t)=\int_0^tg(s)\, \mbox{d}s$;

($H_3$) $\tilde{G}(t):=g(t)t-2G(t)\in C^1(\mathbb{R}, \mathbb{R})$ and
\[
\tilde{G}'(t)t>\left(2+\frac{4}{N}\right)\tilde{G}(t)
\]
for all $t\in\mathbb{R}\setminus\{0\}$.

We mention that in ($H_2$), the number $2_*=2+\frac{4}{N}$ is called the critical mass exponent for normalized solutions, while
$2^*$ is the critical Sobolev exponent of the embedding $H^{1}({\mathbb{R}^N})\hookrightarrow
L^q({\mathbb{R}^N})$, i.e. $2^*=2N/(N-2)$ for $N\geq3$ and $2^*=+\infty$ for $N=1,2$.

Jeanjean proved in \cite{Jean97} that there exists a pair $(\lambda_c, u_c)\in \mathbb{R}\times H^1(\mathbb{R}^N)$ with $\|u_c\|_{L^2(\mathbb{R}^N)}=c$, provided that $N\geq2$ and ($H_1$)-($H_2$) hold, or $N\geq1$ and ($H_1$)-($H_3$)hold.
As for the multiplicity of solutions, assuming that $N\geq2$ and $f$ satisfying ($H_1$)-($H_2$), Bartsch and de Valeriola proved in \cite{BaVa13} the existence of infinitely many radially symmetric solutions via the fountain theorem.
Later, Ikoma and Tanaka gave another proof of multiplicity in \cite{IkTa19} via an idea related to symmetric mountain pass theorems. 
Recently, the study of normalized solutions has been extended to problems with critical $2^*$ growth. 
One can refer to Wei and Wu \cite{WeWu22} for nonlinear terms of the form $f(u)=\mu|u|^{q-2}u+|u|^{2^*-2}$ with $2<q<2_*$ or $2_*\leq q<2^*$, and to Qi and Zou \cite{QiZo23} for $f(u)=|u|^{q-2}u+|u|^{p-2}$ with $2<p<q\leq2^*$ or $2_*<p\leq q=2^*$ or $2_*<p<q\leq 2^*$. 
For other problems involving normalized solutions, we also refer to \cite{ApSe21} for fractional nonlinear Schr\"{o}dinger equations and to \cite{Soa201,Soa202} for modified nonlinear Schr\"{o}dinger equations. 
We mention here that there are numerous references on this subject, here we list only a few relevant ones.

Another interesting problem is the behavior of solutions related to the constraint $m$. 
In \cite{JeLu20}, Jeanjean and Lu not only proved the existence of ground states, but also revealed the basic behavior of the ground state energy $E_m$ as $m > 0$ varies. 
Under proper assumptions on the nonlinear term $f$, they showed that the function $m\to E_m$ is positive, continuous, non-increasing and that $\lim_{m\to0+}E_m=+\infty$, among other properties. 
Recently, Li and Zou \cite{LiZo24} extended this problem to equations involving a critical $2^*$ growth term and showed similar properties for $m\to E_m$.

Inspired by \cite{Jean97,JeLu20,LiZo24}, in this paper we focus on normalized solutions to equations with principal operator of the form 1-Laplacian and the nonlinear perturbation at critical growth, that is, equations of the form (\ref{eq:1-Lap}).

Let $F(u)=\int_0^uf(s)\, \mbox{d}s$. 
We assume that the function $f$ is of class $C^0$ and satisfies the following conditions:

($f_1$) $\lim_{t\to0}\frac{f(t)}{|t|^{\frac{1}{N}}}=0$;

($f_2$) $\lim_{|t|\to\infty}\frac{f(t)}{|t|^{1^*-1}}=0$;

($f_3$) there exists $\theta\in(1+\frac{1}{N},1^*)$ such that $f(t)t\geq\theta F(t)>0$, $\forall t\in\mathbb{R}\setminus\{0\}$;


($f_4$) the map $t\mapsto \frac{\widetilde{F}(t)}{|t|^{1+\frac{1}{N}}}$ is strictly decreasing on $(-\infty,0)$ and strictly increasing on $(0,+\infty)$, where $\widetilde{F}(t)=f(t)t-F(t)$;

($f_5$) $f(t)t<1^*F(t)$, $\forall t\neq0$;

($f_6$) $\lim_{t\to0}\frac{f(t)t}{|t|^{1^*}}=+\infty$.

We mention that the number $1_*=1+\frac{1}{N}$ is the critical mass exponent for normalized solutions of the 1-Laplace equation. 
By ($f_3$) we have
\begin{eqnarray}\label{ineq:f-F-1}
f(t)t>\left(1+\frac{1}{N}\right)F(t)>0,\quad\forall t\in\mathbb{R}\setminus\{0\}.
\end{eqnarray}
Moreover, by ($f_1$)-($f_2$), we can prove that for any $\varepsilon>0$, there exists $C_\varepsilon>0$ such that
\begin{eqnarray}\label{lem:2-3-1}
|f(t)|\leq\varepsilon|t|^{\frac{1}{N}}+C_\varepsilon|t|^{1^*-1}
\end{eqnarray}
and
\begin{eqnarray}\label{lem:2-3-2}
|f(t)|\leq\varepsilon|t|^{1^*-1}+C_\varepsilon|t|^{\frac{1}{N}}.
\end{eqnarray}

Let $\|\cdot\|_r$ denote the norm of $L^r({\mathbb{R}}^N)$ for $1\leq r<+\infty$,
i.e. $\|u\|_r=(\int_{{\mathbb{R}}^N}|u|^r \, \mbox{d}x)^{\frac{1}{r}}$,
and let $\|\cdot\|_{\infty}$ denote the norm of $L^{\infty}({\mathbb{R}}^N)$.
We say that $u\in L^1({\mathbb{R}}^N)$ is a function of bounded variation in ${\mathbb{R}}^N$
if the total variation of $u$ in ${\mathbb{R}}^N$
\begin{eqnarray}\label{def:TV}
|Du|({\mathbb{R}}^N)=\sup\bigg\{\int_{{\mathbb{R}}^N} u{\rm div}\varphi \, \mbox{d}x: \varphi\in C_c^1({\mathbb{R}}^N,{\mathbb{R}}^N),\|\varphi\|_\infty\leq1\bigg\}
\end{eqnarray}
is finite. 
The space of all such functions 
is denoted by $BV({\mathbb{R}}^N)$,
which is a Banach space endowed with the norm
\[
\|u\|_{BV}=|Du|({\mathbb{R}}^N)+\|u\|_1.
\]

According to \cite[Chapter 3]{AFP00}, the embedding $BV({\mathbb{R}}^N)\hookrightarrow L^p({\mathbb{R}}^N)$ is continuous for
$1\leq p\leq 1^*=N/(N-1)$ and is locally compact for $1\leq p<1^*$. 
Hence for $u\in BV({\mathbb{R}}^N)$, the integrals $\int_{{\mathbb{R}}^N}F(u)\mbox{d}x$, $\int_{{\mathbb{R}}^N}f(u)u \, \mbox{d}x$ and $\int_{{\mathbb{R}}^N}|u|^{1^*}\mbox{d}x$ are well defined.

One important tool in the study of normalized solutions is the Pohozaev identity. 
Therefore, in this paper we introduce a Pohozaev-type identity for the 1-Laplace equation (\ref{eq:1-Lap}). 
To be more general, we assume 

($f_0$) there exist constants $q\in(1,1^*)$ and $C>0$ such that
\[
|f(t)|\leq C(1+|t|^{q-1}),\quad \forall t\in\mathbb{R}.
\]

\begin{thm}\label{thm:Poho}
Assume that ($f_0$) holds and that $u$ is a solution to (\ref{eq:1-Lap}) with $u\in W^{1,1}(\mathbb{R}^N)$ and $x\cdot Du\in W^{1,1}(\mathbb{R}^N)$. 
Then the following Pohozaev-type equality holds:
\begin{eqnarray*}\label{eq:Poho}
\int_{\mathbb{R}^N}|Du|-N\int_{\mathbb{R}^N}\left(f(u)u-F(u)\right) \, \mbox{d}x-\int_{\mathbb{R}^N}|u|^{1^*} \, \mbox{d}x=0.
\end{eqnarray*}
\end{thm}

We will use the above Pohozaev identity to investigate properties of the critical values of the functional associated with the equation.

For any given $m>0$, we are interested in the existence of solutions to (\ref{eq:1-Lap})-(\ref{eq:constraint}). 
According to the Lagrange multiplier rule, we know that solutions to \eqref{eq:1-Lap} correspond to the critical points of the functional $I:BV(\mathbb{R}^N)\to\mathbb{R}$ given by
\begin{eqnarray*}\label{eq:1-Lap-funct}
I(u)=\int_{\mathbb{R}^N}|Du|-\int_{\mathbb{R}^N}F(u) \, \mbox{d}x-\frac{1}{1^*}\int_{\mathbb{R}^N}|u|^{1^*}\mbox{d}x
\end{eqnarray*}
under the constraint $\|u\|_1=m$.

We define
\begin{eqnarray}\label{def:P(u)}
P(u)=\int_{\mathbb{R}^N}|Du|-N\int_{\mathbb{R}^N}\left(f(u)u-F(u)\right) \, \mbox{d}x-\int_{\mathbb{R}^N}|u|^{1^*} \, \mbox{d}x
\end{eqnarray}
and
\begin{eqnarray}\label{def:manifold-Sm}
S_m=\left\{u\in BV(\mathbb{R}^N):\int_{\mathbb{R}^N}|u| \, \mbox{d}x=m\right\}.
\end{eqnarray}
For any $m>0$, we define the Pohozaev manifold by
\begin{eqnarray*}\label{def:manifold-Pm}
\mathcal{P}_m=\{u\in S_m:P(u)=0\}
\end{eqnarray*}
and define
\begin{eqnarray*}\label{def:Em}
E_m=\inf_{u\in\mathcal{P}_m}I(u).
\end{eqnarray*}

Let $u\in BV({\mathbb{R}}^N)\cap L^{1^*}({\mathbb{R}}^N)$ and $z\in L^{\infty}({\mathbb{R}}^N,{\mathbb{R}}^N)$
with $\mbox{div}z\in L^{N}({\mathbb{R}}^N)$, here $\mbox{div}$ denotes the divergence in the distributional sense.
We define the functional $(z,Du):C_0^{\infty}({\mathbb{R}}^N)\to{\mathbb{R}}$ by
\begin{eqnarray}\label{def:z-Du}
\langle(z,Du),\varphi\rangle=-\int_{{\mathbb{R}}^N}u\varphi \, \mbox{div}z \, \mbox{d}x-\int_{{\mathbb{R}}^N}uz\cdot\nabla\varphi \,  \mbox{d}x.
\end{eqnarray}
Then $(z,Du)$ is a Radon measure in ${\mathbb{R}}^N$ and satisfies the following Gauss-Green formula \cite{AFP00}:
\begin{eqnarray}\label{Green}
\int_{{\mathbb{R}}^N}u\mbox{div}z\mbox{d}x+\int_{{\mathbb{R}}^N}(z,Du)=0.
\end{eqnarray}

For $u\in BV({\mathbb{R}}^N)$, let ${\mathcal{A}}(u)$ be the set of $u^*\in L_{N,\infty}({\mathbb{R}}^N)$ (defined by \eqref{def:L_N-00} in Section 2) such that
there exists some $z\in L^{\infty}({\mathbb{R}}^N,{\mathbb{R}}^N)$ with $\|z\|_\infty\leq1$, $-\mbox{div}z=u^*$, $(z,Du)=|Du|$.
Let ${\mathcal{B}}(u)$ be the set of $\gamma\in L^{\infty}({\mathbb{R}}^N)$ with $\|\gamma\|_\infty\leq1$ and $\gamma u=|u|$ a.e. in ${\mathbb{R}}^N$.
According to \cite{KaSc07}, to find solutions of the problem \eqref{eq:1-Lap}, it suffices to find $0\neq u\in BV({\mathbb{R}}^N)$ such that
\begin{eqnarray}\label{eq:z}
    \exists u^*\in {\mathcal{A}}(u),\gamma\in {\mathcal{B}}(u)\mbox{ such that}:\quad
    u^*+\lambda\gamma=f(u)+|u|^{1^*-2}u.
\end{eqnarray}
In this sense, we say that $u\in BV({\mathbb{R}}^N)$ is a solution to (\ref{eq:1-Lap}) if \eqref{eq:z} holds. 
Furthermore, if $u$ solves (\ref{eq:1-Lap})-(\ref{eq:constraint}) and $I(u)=\inf_{v\in S_m}I(v)$, we say that $u$ is a normalized ground state solution to (\ref{eq:1-Lap}).

Conditions ($f_1$)-($f_4$) are used to prove the existence of a minimizing Palais-Smale-type sequence for the functional $I$ on the constraint set $S_m$.
Condition ($f_5$) is used to recover the compactness of the minimizing sequence, while imposing condition ($f_6$) on the nonlinearity ensures that the functional has a critical value below the threshold, so that \eqref{eq:1-Lap}-\eqref{eq:constraint} has a positive normalized solution. 
More precisely, we show in Section \ref{sec:exist} that

\begin{thm}\label{thm:m1}
Assume that ($f_1$) to ($f_6$) hold. 
Then there exists $m_0>0$ such that problem (\ref{eq:1-Lap})-(\ref{eq:constraint}) has a ground state solution $u\in BV(\mathbb{R}^N)$ for any $m \geq m_0$.
Moreover, $I(u)=E_m$ and $P(u)=0$.
\end{thm}

Theorem \ref{thm:m1} shows that each weak solution of \eqref{eq:1-Lap}-\eqref{eq:constraint} satisfies the Pohozaev equality given in Theorem \ref{thm:Poho} and thus belongs to the Pohozaev manifold $\mathcal{P}_m$. It also shows that
\[
\inf_{v\in S_m}I(v)=\inf_{v\in\mathcal{P}_m}I(v)=E_m.
\]

As a special case of \eqref{eq:1-Lap}, we consider $f(u)=|u|^{q-2}u$ with $1+\frac{1}{N}<q<1^*$. We have

\begin{thm}\label{thm:m3}
There exists $m_0>0$ such that the following problem has a ground state solution $u\in BV(\mathbb{R}^N)$ for any $m\geq m_0$:
\begin{eqnarray}\label{eq:1-Lap-u^q}
\left\{
  \begin{array}{ll}
    -\mbox{div}\left(\frac{Du}{|Du|}\right)+\lambda \frac{u}{|u|}=f(u)+|u|^{1^*-2}u, & \hbox{in $\mathbb{R}^N$;} \\
    \int_{\mathbb{R}^N}|u| \, \mbox{d}x=m. & \hbox{}
  \end{array}
\right.
\end{eqnarray}

\end{thm}

For the study of the 1-Laplace operator, we would like to mention some literature.
In \cite{KaSc07}, Kawohl and Schuricht provided a reasonable definition of the 1-Laplace operator, demonstrating that solving the 1-Laplace equation is essentially equivalent to solving a differential embedding problem, and then investigated the associated eigenvalue problem. 
In \cite{DeMa09}, Degiovanni and Magrone proved a linking theorem for nonsmooth functionals and employed it to study the existence of solutions to the Br\'{e}zis-Nirenberg problem concerning the 1-Laplace operator. 
Chang, in \cite{Chang09}, proved the existence of a sequence of eigenvalues and eigenfunctions for the 1-Laplace operator via critical point theorems for nonsmooth functionals in \cite{CoDM93,DeMa94}, and presented explicit forms of (strong) eigenvalues and (strong) eigenfunctions in one-dimensional cases. 
In \cite{FiPi18}, Figueiredo and Pimenta investigated the existence of solutions in the entire space for problems involving vanishing potentials, while in \cite{AlPi17}, Alves and Pimenta proved existence and concentration of solutions for parameter-dependent problems. 
In \cite{LiL21}, Li and Liu also applied critical point theorems for nonsmooth functionals to investigate the existence and concentration of solutions for parameter-dependent problems, and established exponential decay properties of the solutions.

The first difficulty is that $I:BV(\mathbb{R}^N)\to\mathbb{R}$ is not $C^1$. 
It is the difference of a convex locally Lipschitz functional and a smooth one, and it is not G\^{a}teaux differentiable in all directions in $BV(\mathbb{R}^N)$. 
In this situation, we need to define its critical points when searching for weak solutions of the equation by variational methods.

The second difficulty is that the dual space of the BV space remains unclear, which poses significant challenges to the creative application of variational methods in studying solutions. 
This is because most such methods require the use of the G\^{a}teaux derivative of the functional, which theoretically belongs to the dual space of the BV space. 
To address this issue, Kawohl and Schuricht proposed an approach in \cite{KaSc07} by extending the domain of the functional from the BV space to the $L^p$ space, transforming it into a lower semicontinuous functional. 
In this case, if the G\^{a}teaux derivative exists, it belongs to the $L^{p'}$ space. 
A drawback of this method is that proving compactness of Palais-Smale sequences for lower semicontinuous functionals is extremely difficult. 
Employing the theory in \cite{Anze85}, Alves and Pimenta in \cite{AlPi17} identified certain G\^{a}teaux differentiable directions for the functional. 
With these directions, combined with variational methods and the subdifferential theory of convex functionals, one can establish related new variational theories, thereby providing a way to resolve the compactness issue for PS sequences.

In Section \ref{sec:exist}, we construct a minimizing Palais-Smale sequence for the constrained functional $I|_{S_m}$ on the manifold $S_m$. This is a key step in the study.
Note that an important difference between the problem studied here and the $p$-Laplace problem with $p>1$ is that the manifold constrained by the $L^1$-norm is only Lipschitz continuous, not $C^1$. 
Therefore, not all elements have tangent spaces. 
This makes it impossible to use the method of constructing an approximate minimizing Palais-Smale sequence in reference \cite{JeLu20}, so new methods are used.
Inspired by \cite{Jean97}, we present an approach for constructing a minimizing Palais-Smale sequence of the constrained functional and to recover the compactness of this sequence. 
Our method is based on the framework of nonsmooth analysis in \cite{Szul86} and the proof process is relatively abstract, see Lemma \ref{lem:4-3}. 
Due to the nonsmoothness of the constrained manifold, this is a difficult task.
During the proof, we can clearly discern the geometric meaning of the Pohozaev identity.
To our knowledge, this is the first paper dealing with normalized solutions under a Lipschitz continuous constraint.

This paper is organized as follows. In Section \ref{sec:Pre}, we recall some mathematical tools concerning BV space and some related ideas on variational methods for nonsmooth functionals, and then prove the Pohozaev-type equality; 
in Section \ref{sec:exist}, we construct a minimizing Palais-Smale sequence and recover its compactness; 
in Section \ref{sec:Em}, we derive some properties of the constrained critical values of the functional, which are related to the proof of the main theorems, and then prove Theorem \ref{thm:m1}.

\section{Mathematical tools and variational framework}\label{sec:Pre}
\indent\indent
In the previous section, we provided the definition of the BV space.
In this section, we extend the domain of the functional to a larger function space, making it lower semicontinuous. 
We will discuss the relationship between the G\^{a}teaux derivative and the subdifferential of the functional, and then combine the subdifferential method to provide an approach to finding the critical points of the functional. 
This idea is closely related to that in \cite{KaSc07}, see also \cite{AlPi17,LiL21}.

Recall from \cite[Chapter 3]{AFP00} that $C_0^\infty(\mathbb{R}^N)$ is not dense in $BV(\mathbb{R}^N)$ with respect to strong convergence, since the closure of $C_0^\infty(\mathbb{R}^N)$ in the BV norm is equal to $W^{1,1}(\mathbb{R}^N)$, which is a proper subspace of $BV(\mathbb{R}^N)$. 
This fact leads to many difficulties, rendering the usual methods for studying p-Laplace equations ineffective. 
For instance, the most commonly used scaling map $t\to u({\cdot}/{t})$ may be discontinuous in BV space. 
However, there is a weaker convergence, called strict convergence \cite[Definition 3.14]{AFP00} (or intermediate convergence, cf. \cite{AlPi17}), which is more convenient in applications. 
We say that a sequence $\{u_n\}\subset BV(\mathbb{R}^N)$ converges to $u\in BV(\mathbb{R}^N)$ strictly if 
\[
u_n\to u,\quad\mbox{in~} L^1(\mathbb{R}^N)
\]
and 
\[
\int_{\mathbb{R}^N}|Du_n|\to\int_{\mathbb{R}^N}|Du|
\]
as $n\to\infty$. 
Moreover, $C_0^{\infty}(\mathbb{R}^N)$ is dense in $BV(\mathbb{R}^N)$ with respect to strict convergence.

According to the Radon-Nikod\'{y}m theorem (cf. \cite[Theorem 4.2.1]{ABM06}), for $u\in BV(\mathbb{R}^N)$, $Du$ can be decomposed into an absolutely continuous part $D^a u$ and a singular part $D^s u$, i.e. $Du=D^a u+D^s u$, and the (Lebesgue) derivative of $Du$ with respect to $|Du|$ is given by 
\[
\frac{Du}{|Du|}(x)=\lim_{\rho\to0}\frac{Du(B_\rho(x))}{|Du|(B_\rho(x))}.
\]
Here and in what follows, $B_\rho(x)$ denotes the ball centered at $x\in\mathbb{R}^N$ with radius $\rho>0$.

We define $J:BV(\mathbb{R}^N)\to\mathbb{R}$ by
\[
J(u)=\int_{\mathbb{R}^N}|Du|+\lambda\int_{\mathbb{R}^N}|u| \, \mbox{d}x.
\]
By definition, $J$ is convex and Lipschitz continuous. 
For a given $u\in BV(\mathbb{R}^N)$, $J$ is not differentiable at $u$ along all directions $v\in BV(\mathbb{R}^N)$. 
But fortunately, according to \cite{Anze85}, $J$ is differentiable at $u$ along directions $v\in BV(\mathbb{R}^N)$ such that $D^s v$ is absolutely continuous with respect to $D^s u$ and such that $v$ is equal to $0$ a.e. on the set where $u$ vanishes, and the directional derivative is (cf. \cite{AlPi17})
\begin{eqnarray}\label{def:J'}
  \langle J'(u),v\rangle &=& \int_{\mathbb{R}^N}\frac{Du^aD^a v}{|D^a u|} \, \mbox{d}x
+\int_{\mathbb{R}^N}\frac{Du}{|Du|}(x)\frac{Dv}{|Dv|}(x)|Dv|^s
+\lambda\int_{\mathbb{R}^N}\mbox{sgn}(u)v \,\mbox{d}x,
\end{eqnarray}
where $\mbox{sgn}(u(x))=0$ if $u(x)=0$ and $\mbox{sgn}(u(x))=u(x)/|u(x)|$ if $u(x)\neq0$. 
In particular,
$J'(u)u=J(u)$.

Let $\varphi\in C_0^\infty(\mathbb{R}^N)$, $u\in BV(\mathbb{R}^N)$ and $v=\varphi u$. Then $v\in BV(\mathbb{R}^N)$, $Dv=\varphi Du+uD\varphi$ and {$D^s v=\varphi D^su$} is absolutely continuous with respect to $D^s u$. 
Hence, $J$ is differentiable at $u$ along direction $v$ and the Calder\'{o}n-Zygmund Theorem \cite[Theorem 3.83]{AFP00} gives
\begin{eqnarray}\label{def:J'-1}
  \langle J'(u),v\rangle &=& \int_{\mathbb{R}^N}\varphi{|D^a u|}
+\int_{\mathbb{R}^N}\frac{D^auD\varphi}{|D^a u|}u\, \mbox{d}x
+\int_{\mathbb{R}^N}\varphi|Du|^s
+\lambda\int_{\mathbb{R}^N}\, \mbox{sgn}(u)v\, \mbox{d}x\nonumber\\
&=&\int_{\mathbb{R}^N}\varphi{|Du|}
+\int_{\mathbb{R}^N}\frac{D^auD\varphi}{|D^a u|}u \, \mbox{d}x
+\lambda\int_{\mathbb{R}^N} \mbox{sgn}(u)v \, \mbox{d}x.
\end{eqnarray}

For $u\in BV(\mathbb{R}^N)$, denote 
\begin{eqnarray}\label{def:J-1&J-2}
J_1(u)=\int_{\mathbb{R}^N}|Du|, \quad J_2(u)=\int_{\mathbb{R}^N}|u| \, \mbox{d}x
\end{eqnarray}
and
\begin{eqnarray}\label{def:F}
\mathcal{F}(u)=\int_{\mathbb{R}^N}F(u) \, \mbox{d}x+\int_{\mathbb{R}^N}|u|^{1^*} \, \mbox{d}x.
\end{eqnarray}
Then we have $\mathcal{F}\in C^1(BV(\mathbb{R}^N))$ and $I(u)=J_1(u)-\mathcal{F}(u)$. 
We define
\begin{eqnarray}\label{def:I-lambda}
I_\lambda(u)=J(u)-\mathcal{F}(u).
\end{eqnarray}
Let $\partial I_\lambda(u)$ be the Clarke generalized gradient of $I_\lambda$ at $u$ (cf. \cite{Clar76}), and assume that $0\in\partial I_\lambda(u)$. 
Since $I$ is convex on $BV(\mathbb{R}^N)$, $0\in\partial I_\lambda(u)$ implies that
\begin{eqnarray*}\label{ineq:subdiff}
  J(w)-J(u) &\geq& \mathcal{F}'(u)(w-u),\quad\forall w\in BV(\mathbb{R}^N).
\end{eqnarray*}
This means that $0\in\partial I_\lambda(u)$ is equivalent to $\mathcal{F}'(u)\in\partial J_1(u)$.

In order to find $z\in L^{\infty}({\mathbb{R}}^N,{\mathbb{R}}^N)$ such that $\mbox{Div}z\in \mathcal{A}(u)$ (note that $\gamma\in L^{\infty}({\mathbb{R}}^N)$ is much easier to obtain) and \eqref{eq:z} holds, we extend $I_\lambda$ to a larger function space.

Let $X=L^1(\mathbb{R}^N)\cap L^{1^*}(\mathbb{R}^N)$ be endowed with the norm
\[
\|u\|_X=\|u\|_1+\|u\|_{1^*},\quad u\in X.
\]
We extend $I_\lambda, J, \mathcal{F}$ to $X$.
We define $\overline{I}, \overline{J}, \overline{\mathcal{F}}:X\to\mathbb{R}\cup+\infty$ by
\begin{eqnarray*}
\overline{J}_1(u)&=&
\left\{
  \begin{array}{ll}
    J_1(u), & \hbox{if $u\in BV(\mathbb{R}^N)$;} \\
    +\infty, & \hbox{if $u\in X\setminus BV(\mathbb{R}^N)$,}
  \end{array}
\right.\\
\overline{J}_2(u)&=&\int_{\mathbb{R}^N}|u| \, \mbox{d}x,\quad u\in X, \\
\overline{\mathcal{F}}(u)&=&\int_{\mathbb{R}^N}F(u) \, \mbox{d}x+\int_{\mathbb{R}^N}|u|^{1^*} \, \mbox{d}x 
\end{eqnarray*}
and 
\[
\overline{J}=\overline{J}_1+\lambda\overline{J}_2,\quad
\overline{I}_\lambda=\overline{J}-\overline{\mathcal{F}}.
\]
By definition, $\overline{\mathcal{F}}\in C^1(X)$, $\overline{J}$ is convex and lower semicontinuous on $X$. 
According to \cite{AlPi17},
\[
\partial\overline{J}(u)=\partial\overline{J}_1(u)+\lambda\partial\overline{J}_2(u),\quad \forall u\in X.
\] 
Moreover, for $u\in BV(\mathbb{R}^N)$, $0\in\partial \overline{I}_\lambda(u)$ implies that
\begin{eqnarray*}\label{ineq:subdiff-X}
  \overline{J}(w)-\overline{J}(u) &\geq& \overline{\mathcal{F}}'(u)(w-u),\quad\forall w\in X
\end{eqnarray*}
and $\overline{\mathcal{F}}'(u)\in\partial \overline{J}(u)$.
Therefore, by Lemma 3 in \cite{AlPi17}, if $u\in BV(\mathbb{R}^N)$ is such that $0\in\partial {I}_\lambda(u)$,  then we also have $0\in\partial \overline{I}_\lambda(u)$ since $\overline{I}_\lambda=+\infty$ on $X\setminus BV(\mathbb{R}^N)$. 
This also implies that if $u\in BV(\mathbb{R}^N)$ is such that $\mathcal{F}'(u)\in\partial {J}(u)$, then $\overline{\mathcal{F}}'(u)\in\partial \overline{J}(u)$.

Note that the dual space $X'$ of $X$ is equivalent to $L_{\infty,N}(\mathbb{R}^N)$, where
\begin{eqnarray}\label{def:L_N-00}
L_{\infty,N}(\mathbb{R}^N)=\left\{g:\mathbb{R}^N\to\mathbb{R}~ \mbox{measurable}: \|g\|_{\infty,N}<+\infty\right\},
\end{eqnarray}
with norm
\[
\|g\|_{\infty,N}=\sup_{\|\phi\|_1+\|\phi\|_{1^*}\leq1}\left|\int_{\mathbb{R}^N}g\phi \, \mbox{d}x\right|.
\]
If $\overline{\mathcal{F}}'(u)\in\partial \overline{J}(u)$, then there exist $u^*\in \partial\overline{J}_1(u)$ and $\gamma^*\in \partial\overline{J}_2(u)$ such that
\[
\overline{\mathcal{F}}'(u)=u^*+\lambda\gamma^*\quad\mbox{in~}L_{\infty,N}(\mathbb{R}^N).
\]
Then following Proposition 4.23 in \cite{KaSc07}, there exist $z\in L^{\infty}(\mathbb{R}^N,\mathbb{R}^N)$ such that $\|z\|_\infty\leq1$, $(z,Du)=|Du|$, $-\mbox{div}z=u^*$ in $L_{\infty,N}(\mathbb{R}^N)$ and $\gamma\in L^\infty(\mathbb{R}^N)$ with $\|\gamma\|_\infty\leq1$, $\gamma u=|u|$ a.e. on $\mathbb{R}^N$. 
In conclusion, we have
\begin{eqnarray}\label{eq:1-Lap-z}
-\mbox{div}z+\lambda\gamma=f(u)+|u|^{1^*-2}u,\quad\mbox{in}~L_{\infty,N}(\mathbb{R}^N).
\end{eqnarray}
Thus, finding a solution of \eqref{eq:1-Lap-z} reduces to finding a $u\in BV(\mathbb{R}^N)$ with $J_2(u)=m$ and a $\lambda$ such that $0\in\partial I_\lambda(u)$, or equivalently $\mathcal{F}'(u)\in\partial J(u)$, in the topology of $BV(\mathbb{R}^N)$.

Since $BV({\mathbb{R}}^N)\subset L^{1}({\mathbb{R}}^N)$, for $u\in BV({\mathbb{R}}^N)\cap L^{1^*}({\mathbb{R}}^N)$ and $z\in L^{\infty}({\mathbb{R}}^N,{\mathbb{R}}^N)$
with $\mbox{div}z\in L_{\infty,N}({\mathbb{R}}^N)$ (divergence in the distributional sense), 
the definition of the measure $(z,Du):C_0^{\infty}({\mathbb{R}}^N)\to{\mathbb{R}}$ given by (\ref{def:z-Du}) is still valid and is a Radon measure also; moreover, the Green-Gauss formula (\ref{Green}) also holds.

At the end of this section, we establish the Pohozaev identity.
The proof is standard. For simplicity, we temporarily assume $\lambda=1$ in the proof. 
We remark that for problems without a critical growth term, a version of the identity on a domain $\Omega\subset\mathbb{R}^N$ with Lipschitz boundary was proved in \cite{SaLe18}, and a version of the identity in the entire space for radial solutions was proved in \cite{ChPi21}.

\begin{proof}[\it Proof of Theorem \ref{thm:Poho}]
Since $x\cdot Du\in W^{1,1}(\mathbb{R}^N)$, we have
\[
D(x\cdot Du)=Du+x\cdot D^2u\quad \mbox{in~}\mathcal{D}'(\mathbb{R}^N),
\]
where $D^2u$ denotes the second derivatives of $u$ in the distributional sense and
\[
(x\cdot D^2u)_j=\sum_{i=1}^Nx_i\frac{\partial^2u}{\partial x_i\partial x_j}.
\]
By assumption, we know that
\[
(x\cdot D^2u)_j\in L^1(\mathbb{R}^N),\quad j=1,\cdots,N.
\]
By Stampacchia's theorem, $D(x\cdot Du)=0$ a.e. on the set $\{x\cdot Du=0\}$. This implies that
\[
(x\cdot D^2u)_j=0\quad\mbox{a.e. in~}\{|Du|=0\}.
\]
Since $u\in W^{1,1}(\mathbb{R}^N)$ is a solution to (\ref{eq:1-Lap}), there exist $z\in L^{\infty}(\mathbb{R}^N,\mathbb{R}^N)$ with $\|z\|_\infty\leq1$, $-\mbox{div}z\in L_{\infty,N}(\mathbb{R}^N)$, $(z,Du)=|Du|$ and a $\gamma\in L^{\infty}(\mathbb{R}^N)$ with $\|\gamma\|_\infty\leq1$, $\gamma u=|u|$ such that (\ref{eq:z}) holds. Note that $x\cdot Du\in W^{1,1}(\mathbb{R}^N)$, we derive that
\begin{eqnarray}\label{thm:Poho-1}
\int_{\mathbb{R}^N}-(x\cdot Du) \, \mbox{div}z \, \mbox{d}x
&=&\int_{\mathbb{R}^N}(z,D(x\cdot Du)) \, \mbox{d}x\nonumber\\
&=&\int_{\mathbb{R}^N}(z\cdot Du)+\int_{\mathbb{R}^N}(z,x\cdot D^2u) \, \mbox{d}x\nonumber\\
&=&\int_{\mathbb{R}^N}|Du|+\int_{\mathbb{R}^N}(z,x\cdot D^2u) \, \mbox{d}x.
\end{eqnarray}
On the one hand, since $u\in W^{1,1}(\mathbb{R}^N)$ (we may assume that $|Du|>0$), we have $D(|Du|)=\frac{Du}{|Du|}\cdot D^2u=z\cdot D^2u$ in the weak sense, so that
\[
\int_{\mathbb{R}^N}x\cdot D(|Du|) \, \mbox{d}x=\int_{\mathbb{R}^N}x\cdot(z\cdot D^2u) \, \mbox{d}x=\int_{\mathbb{R}^N}(z,(x,D^2u)) \, \mbox{d}x.
\]
On the other hand,
\[
\int_{\mathbb{R}^N}x\cdot D(|Du|) \, \mbox{d}x=\int_{\mathbb{R}^N}(-\mbox{div}x)|Du|=-N\int_{\mathbb{R}^N}|Du|.
\]
Therefore,
\[
\int_{\mathbb{R}^N}(z,(x\cdot D^2u)) \, \mbox{d}x=-N\int_{\mathbb{R}^N}|Du|.
\]
Then from (\ref{thm:Poho-1}), we get
\begin{eqnarray}\label{thm:Poho-2}
\int_{\mathbb{R}^N}-(x\cdot Du) \, \mbox{div}z \, \mbox{d}x=-(N-1)\int_{\mathbb{R}^N}|Du|.
\end{eqnarray}

Now we choose $x\cdot Du\in W^{1,1}(\mathbb{R}^N)$ as a test function in (\ref{eq:z}) and integrate over $\mathbb{R}^N$. 
Note that $u\in L^1(\mathbb{R}^N)\cap L^{1^*}(\mathbb{R}^N)$ and $Du\in L^1(\mathbb{R}^N)$, we can apply the Green formula to obtain that
\begin{eqnarray*}
&&\int_{\mathbb{R}^N}-(x\cdot Du) \, \mbox{div}z \, \mbox{d}x+\int_{\mathbb{R}^N}\frac{u}{|u|}(x\cdot Du) \, \mbox{d}x\\
&=&\int_{\mathbb{R}^N}f(u)(x\cdot Du) \, \mbox{d}x+\int_{\mathbb{R}^N}|u|^{1^*-2}u(x\cdot Du) \, \mbox{d}x\\
&=&\int_{\mathbb{R}^N}\sum_{i=1}^{N}x_i\frac{\partial F(u)}{\partial x_i} \, \mbox{d}x+\frac{1}{1^*}\int_{\mathbb{R}^N}\sum_{i=1}^Nx_i\frac{\partial(|u|^{1^*})}{\partial x_i} \, \mbox{d}x\\
&=&-N\int_{\mathbb{R}^N}F(u) \, \mbox{d}x-\frac{N}{1^*}\int_{\mathbb{R}^N}|u|^{1^*} \, \mbox{d}x.
\end{eqnarray*}
Inserting (\ref{thm:Poho-2}) into the above equality and note that $\frac{u}{|u|}Du=D(|u|)$, we get
\begin{eqnarray}\label{thm:Poho-3}
(N-1)\int_{\mathbb{R}^N}|Du|+N\int_{\mathbb{R}^N}|u| \, \mbox{d}x
=N\int_{\mathbb{R}^N}F(u) \, \mbox{d}x+(N-1)\int_{\mathbb{R}^N}|u|^{1^*} \, \mbox{d}x.
\end{eqnarray}
Furthermore, using $u$ as a test function in (\ref{eq:1-Lap}), we also get
\begin{eqnarray}\label{thm:Poho-4}
\int_{\mathbb{R}^N}|Du|+\int_{\mathbb{R}^N}|u| \, \mbox{d}x=
\int_{\mathbb{R}^N}f(u)u \, \mbox{d}x+\int_{\mathbb{R}^N}|u|^{1^*} \, \mbox{d}x.
\end{eqnarray}
From (\ref{thm:Poho-3})-(\ref{thm:Poho-4}) we get
\begin{eqnarray*}
\int_{\mathbb{R}^N}|Du|-N\int_{\mathbb{R}^N}(f(u)u-F(u)) \, \mbox{d}x-\int_{\mathbb{R}^N}|u|^{1^*} \, \mbox{d}x=0.
\end{eqnarray*}
This completes the proof of Theorem \ref{thm:Poho}.
\end{proof}

\begin{rem}\label{rem:2-1}
(i) For the 1-Laplace equation 
\begin{eqnarray}\label{eq:1-Lap-f}
    -\mbox{div}\left(\frac{Du}{|Du|}\right)+\frac{u}{|u|}=f(u), & \hbox{in $\mathbb{R}^N$},
\end{eqnarray}
the Pohozaev identity becomes
\begin{eqnarray}\label{thm:Poho-5}
\int_{\mathbb{R}^N}|Du|-N\int_{\mathbb{R}^N}(f(u)u-F(u)) \, \mbox{d}x=0.
\end{eqnarray}

(ii) From (\ref{thm:Poho-3})-(\ref{thm:Poho-4}), we also have
\begin{eqnarray}\label{thm:Poho-6}
(N-1)\int_{\mathbb{R}^N}f(u)u \, \mbox{d}x-N\int_{\mathbb{R}^N}F(u) \, \mbox{d}x+\int_{\mathbb{R}^N}|u| \, \mbox{d}x=0.
\end{eqnarray}
This equality also holds for problem (\ref{eq:1-Lap-f}).
\end{rem}

Recall that in ($f_0$), we assume that $q\in(1,1^*)$.
For problem \eqref{eq:1-Lap-f}, this requirement is consistent with (\ref{thm:Poho-5})-(\ref{thm:Poho-6}) in order to ensure the solution $u$ is nontrivial. 

Although the Pohozaev identity is established only for solution of (\ref{eq:1-Lap}) in $W^{1,1}(\mathbb{R}^N)$, the functional $P$ given by (\ref{def:P(u)}) is well defined on $BV(\mathbb{R}^N)$.

\section{Construction of a minimizing Palais-Smale sequence}\label{sec:exist}
\indent\indent
In this section, we prove the existence of a sequence (called a Palais-Smale sequence) $\{v_n\}\subset S_m$ for $I|_{S_m}$ at a mountain-pass-type critical level $\gamma_m$, and then study the compactness of this sequence in $L^r(\mathbb{R}^N)$, for all $1\leq r\leq 1^*$. 

To begin with, we define
\begin{eqnarray*}\label{def:D-11}
D^{1,1}(\mathbb{R}^N)=\{u\in L^{1^*}(\mathbb{R}^N): Du\in L^1(\mathbb{R}^N)\}.
\end{eqnarray*}
We also denote by $S$ the best constant for the embedding $D^{1,1}(\mathbb{R}^N)\hookrightarrow L^{1^*}(\mathbb{R}^N)$, i.e.
\begin{eqnarray}\label{def:S1}
S=\inf_{u\in D^{1,1}(\mathbb{R}^N)\setminus\{0\}}\frac{\|Du\|_{1}}{\|u\|_{1^*}}.
\end{eqnarray}

According to \cite[Chapter 3]{AFP00}, 
the embedding $BV(\mathbb{R}^N)\hookrightarrow L^r(\mathbb{R}^N)$ is continuous for $1\leq r\leq 1^*$ and is locally compact for $1\leq r<1^*$.
Thus the Gagliardo-Nirenberg-type inequality for functions $u\in BV(\mathbb{R}^N)$ holds, that is, there exists a constant $C_r>0$, depending on $r$, such that
\begin{eqnarray}\label{ineq:G-N}
\|u\|_r\leq C_r\|u\|_1^{1-\theta_r}\|Du\|_1^{\theta_r},
\end{eqnarray}
where $\theta_r=N\left(1-\frac{1}{r}\right)$.

We begin by proving that $I|_{S_m}$ has the mountain pass geometry on $S_m$ as in \cite{Jean97}.
We recall the following scaling map given in \cite{Jean97}.

For $w\in BV(\mathbb{R}^N)$ and $\tau\in\mathbb{R}$, we define $\tau* w$ by
\begin{eqnarray*}
(\tau*w)(x)=e^{N\tau}w(e^{\tau}x).
\end{eqnarray*}
By direct computation, we have
\begin{eqnarray*}
\int_{\mathbb{R}^N}|D(\tau*w)|=e^\tau\int_{\mathbb{R}^N}|Dw|,\quad\mbox{and}\quad
\int_{\mathbb{R}^N}|\tau*w|^r \, \mbox{d}x=e^{(r-1)N\tau}\int_{\mathbb{R}^N}|w|^r \, \mbox{d}x.
\end{eqnarray*}
Moreover,
\begin{eqnarray*}
w\in S_m\Leftrightarrow \tau*w\in S_m,
\end{eqnarray*}
where $S_m$ is defined in (\ref{def:manifold-Sm}).
As a result, we have
\begin{eqnarray}\label{eq:I(tau-u)}
I(\tau*u)&=&\int_{\mathbb{R}^N}|D(\tau*u)|-\int_{\mathbb{R}^N}F(\tau*u) \, \mbox{d}x-\frac{1}{1^*}\int_{\mathbb{R}^N}|\tau*u|^{1^*} \, \mbox{d}x\nonumber\\
&=&e^{\tau}\int_{\mathbb{R}^N}|Du|-e^{-N\tau}\int_{\mathbb{R}^N}F(e^{N\tau}u) \, \mbox{d}x-\frac{e^{1^*\tau}}{1^*}\int_{\mathbb{R}^N}|u|^{1^*} \, \mbox{d}x.
\end{eqnarray}

\begin{lem}\label{lem:2-5}
  Suppose that ($f_1$)-($f_4$) hold. Then $\forall u\in BV(\mathbb{R}^N)\setminus\{0\}$, we have
  
  (1) $I(\tau*u)\to0^+$ as $\tau\to-\infty$;
  
  (2) $I(\tau*u)\to-\infty$ as $\tau\to+\infty$.
\end{lem}
\begin{proof}
  By (\ref{ineq:f-F-1}), we have $F(t)\geq0$, $\forall t\in\mathbb{R}$. Thus by (\ref{eq:I(tau-u)}),
  \begin{eqnarray*}
    I(\tau*u) &\leq& e^{\tau}\left(\int_{\mathbb{R}^N}|Du|-\frac{1}{1^*}e^{(1^*-1)\tau}\int_{\mathbb{R}^N}|u|^{1^*} \, \mbox{d}x\right).
  \end{eqnarray*}
  Since $e^\tau\to0$ as $\tau\to-\infty$, we have $I(\tau*u)>0$ and $I(\tau*u)\to0^+$ as $\tau\to-\infty$. Thus conclusion (1) holds.
  
  Similarly, since $e^\tau\to+\infty$ as $\tau\to+\infty$, we have $I(\tau*u)\to-\infty$ as $\tau\to+\infty$, so conclusion (2) holds.
\end{proof}

\begin{lem}\label{lem:4-1}
Suppose ($f_1$)-($f_2$) hold. Then there exists a constant $K_m>0$ such that 
\[
0<\sup_{u\in A}I(u)<\inf_{u\in B}I(u),
\]
where 
$
  A=\{u\in S_m: \int_{\mathbb{R}^N}|Du|\leq K_m\},
  B=\{u\in S_m: \int_{\mathbb{R}^N}|Du|=2K_m\}.
$
\end{lem}
\begin{proof}
For $m>0$, we define $B_m=\{u\in BV(\mathbb{R}^N):\|u\|_1\leq m\}$.
Then $\forall u\in B_m$, by (\ref{lem:2-3-1}), G-N inequality (\ref{ineq:G-N}) and the embedding theorem (\ref{def:S1}), we have
\begin{eqnarray*}
  |\mathcal{F}(u)|
  &:=&\left|\int_{\mathbb{R}^N}F(u) \, \mbox{d}x+\frac{1}{1^*}\int_{\mathbb{R}^N}|u|^{1^*} \, \mbox{d}x\right|\\
  &\leq&\varepsilon\int_{\mathbb{R}^N}|u|^{1+\frac{1}{N}} \, \mbox{d}x+C_\varepsilon\int_{\mathbb{R}^N}|u|^{1^*} \, \mbox{d}x
  +\frac{1}{1^*}\int_{\mathbb{R}^N}|u|^{1^*} \, \mbox{d}x\\
  &\leq&C\varepsilon\|u\|_1^{\frac{1}{N}}\|Du\|_1+\left(C_\varepsilon+\frac{1}{1^*}\right)S^{-1^*}\|Du\|_1^{1^*}\\
  &\leq&C\varepsilon m^{\frac{1}{N}}\|Du\|_1+\left(C_\varepsilon+\frac{1}{1^*}\right)S^{-1^*}\|Du\|_1^{1^*}\\
  &=&\left[C\varepsilon m^{\frac{1}{N}}+\left(C_\varepsilon+\frac{1}{1^*}\right)S^{-1^*}\|Du\|_1^{1^*-1}\right]\|Du\|_1.
\end{eqnarray*}
Taking $\varepsilon=\frac{1}{8Cm^{\frac{1}{N}}}$ and $\delta=\left(\frac{S^{1^*}}{8(C_\varepsilon+\frac{1}{1^*})}\right)^{\frac{1}{1^*-1}}$. Then for $\|Du\|_1\leq\delta$, we have
$|\mathcal{F}(u)|\leq\frac{1}{4}\|Du\|_1$. This also shows that 
\begin{eqnarray}\label{lem:2-3-(1)}
I(u)\geq\frac{1}{4}\int_{\mathbb{R}^N}|Du|
\end{eqnarray}
for each $u\in B_m$ with $\|Du\|_1\leq\delta$.

Let $K>0$ be arbitrary but fixed, and let $u,v\in S_m$ be such that $J_1(u)=K$ and $J_1(v)=2K$, where $J_1(u)=\int_{\mathbb{R}^N}|Du|$. Note that $|\mathcal{F}(v)|\leq\frac{1}{4}J_1(v)$ for $\|v\|_{BV}=2K$ with $2K\leq\delta$. Thus 
\begin{eqnarray*}
I(v)-I(u)&=&J_1(v)-J_1(u)-\mathcal{F}(v)+\mathcal{F}(u)\\
&\geq&J_1(v)-J_1(u)-\frac{1}{4}J_1(v)\\
&\geq&\frac{K}{2}.
\end{eqnarray*}
This implies the conclusion.
\end{proof}

From Lemma \ref{lem:2-5} and Lemma \ref{lem:4-1}, we can easily obtain the following result.

\begin{lem}\label{lem:4-2}
Suppose ($f_1$)-($f_4$) hold. Then there exists $u_1, u_2\in S_m$ such that 
\begin{itemize}
  \item [(1)] $\int_{\mathbb{R}^N}|Du_1|\leq K_m$; 
  \item [(2)] $\int_{\mathbb{R}^N}|Du_2|>2K_m$;
  \item [(3)] $I(u_1)>0\geq I(u_2)$,
\end{itemize}
where $K_m$ is given by Lemma \ref{lem:4-1}. 
\end{lem}

We remark that employing the approach in \cite[Lemma 2.8]{Jean97}, we can prove that $u_1$ and $u_2$ are arc-connected.

We define
\begin{eqnarray*}\label{lem:4-2-4}
  \gamma_m &=& \inf_{h\in\Gamma_m}\max_{t\in[0,1]}I(h(t)),
\end{eqnarray*}
where
\begin{eqnarray*}\label{lem:4-2-5}
  \Gamma_m &=& \{h\in C([0,1],S_m): h(0)=u_1, h(1)=u_2\}.
\end{eqnarray*}
Then $\gamma_m$ is a mountain-pass-type critical value of $I$ on $S_m$; moreover, by Lemma \ref{lem:4-1}, we have
\begin{eqnarray}\label{lem:4-2-6}
  \gamma_m &>& \max\{I(u_1), I(u_2)\}.
\end{eqnarray}

Let ${E}=BV(\mathbb{R}^N)\times\mathbb{R}$ be the product space endowed with the norm
\[
\|(u,s)\|_E=\|u\|_{BV}+|s|_{\mathbb{R}}.
\]
Then $E$ is a Banach space. We define the functional $\widetilde{I}: E\to\mathbb{R}$ by
\begin{eqnarray}\label{def:eq-I~}
  \widetilde{I}((u,\tau)) = I(\tau*u)
=e^\tau\int_{\mathbb{R}^N}|Du|-e^{-N\tau}\int_{\mathbb{R}^N}F(e^{N\tau}u) \, \mbox{d}x
-\frac{e^{1^*\tau}}{1^*}\int_{\mathbb{R}^N}|u|^{1^*} \, \mbox{d}x,\quad \forall (u,\tau)\in E.\nonumber\\
\end{eqnarray}

Set
\begin{eqnarray*}\label{lem:4-2-1}
  \widetilde{\gamma}_m = \inf_{h\in\widetilde{\Gamma}_m}\max_{t\in[0,1]}\widetilde{I}(h(t)),
\end{eqnarray*}
where
\begin{eqnarray*}\label{lem:4-2-2}
\widetilde{\Gamma}_m=\{h\in C([0,1], S_m\times\mathbb{R}): h(0)=(u_1,0),h(1)=(u_2,0)\}.
\end{eqnarray*}
Note that
\begin{eqnarray*}
  I(u_1) &=& I(0*u_1)=\widetilde{I}((u_1,0))\\
I(u_2) &=& I(0*u_2)=\widetilde{I}((u_2,0)),
\end{eqnarray*}
we have $\Gamma_m\subset\widetilde{\Gamma}_m$, and so $\widetilde{\gamma}_m\leq\gamma_m$. 

We claim that $\widetilde{\gamma}_m\geq\gamma_m$, so that $\widetilde{\gamma}_m={\gamma}_m$.
In fact, for any $\widetilde{h}\in\widetilde{\Gamma}_m$, let $\widetilde{h}(t)=(\widetilde{h}_1(t),\widetilde{h}_2(t))\in S_m\times\mathbb{R}$.
Then for all $t\in[0,1]$, we have
\[
\widetilde{I}(\widetilde{h}(t))=\widetilde{I}((\widetilde{h}_1(t),\widetilde{h}_2(t)))=I(\widetilde{h}_2(t)*\widetilde{h}_1(t)).
\]
Since $\widetilde{h}(t)\in\widetilde{\Gamma}_m$, we have $\widetilde{h}_2(0)=0, \widetilde{h}_2(1)=0$. 
Thus by setting $h(t)= \widetilde{h}_2(t)*\widetilde{h}_1(t)$, we have $h(0)=\widetilde{h}_1(0)=u_1$ and $h(1)=\widetilde{h}_1(1)=u_2$. This means that $h\in\Gamma_m$ and 
\[
\max_{t\in[0,1]}\widetilde{I}(\widetilde{h}(t))=\max_{t\in[0,1]}I(h(t)).
\]
This proves the claim.  By \eqref{lem:4-2-6}, we have 
\begin{eqnarray}\label{lem:4-2-3}
\widetilde{\gamma}_m>\max\{\widetilde{I}((u_1,0)),\widetilde{I}((u_2,0))\}:=\widetilde{\gamma}_0>0.
\end{eqnarray}

Now we construct a Palais-Smale sequence for the constrained functional $I|_{S_m}$ at level $\gamma_m$.
We cannot adopt the method from \cite{JeLu20} to prove the next lemma because the manifold $S_m$ is not $C^1$ but merely Lipschitz. Inspired by \cite{Jean97}, we construct this sequence indirectly through the functional $\widetilde{I}$. We will make use of the kernel of $\mathcal{F}'(u)$ instead of the tangent space $T_u$ for $u\in S_m$. Our proof is based on the framework of nonsmooth analysis.

For $(u,\tau)\in E$, we define 
\[
\widetilde{J}_1((u,\tau))=J_1(\tau*u)\quad\mbox{and}\quad \widetilde{J}_2((u,\tau))=J_2(\tau*u),
\]
where $J_1, J_2$ are defined in (\ref{def:J-1&J-2}). 
Similarly, for $(u,\tau)\in E$, we define $\widetilde{\mathcal{F}}((u,\tau))={\mathcal{F}}(\tau*u)$, where ${\mathcal{F}}$ is defined by (\ref{def:F}). Then $\forall (v,s)\in E$ and $\forall \kappa>0$, we have
\begin{eqnarray}\label{def:tideF'}
  \langle\widetilde{\mathcal{F}}'((u,\tau)),\kappa(v,s)\rangle
  &=& \kappa e^{-N\tau}\int_{\mathbb{R}^N}f(e^{N\tau}u)e^{N\tau}v \, \mbox{d}x+\kappa e^{1^*\tau}\int_{\mathbb{R}^N}|u|^{1^*-2}uv \, \mbox{d}x \nonumber\\
  &&+\kappa sNe^{-N\tau}\int_{\mathbb{R}^N}\Big(f(e^{N\tau}u)e^{N\tau}u-F(e^{N\tau}u)\Big) \, \mbox{d}x+\kappa se^{1^*\tau}\int_{\mathbb{R}^N}|u|^{1^*} \, \mbox{d}x.\nonumber\\
\end{eqnarray}
For $(u,\tau)\in E$, we define two continuous mappings from $E$ to $\mathbb{R}$ by
\begin{eqnarray*}\label{def:F'-u}
  \langle \widetilde{\mathcal{F}}_u'((u,\tau)),v\rangle &=& e^{-N\tau}\int_{\mathbb{R}^N}f(e^{N\tau}u)e^{N\tau}v \, \mbox{d}x
  +e^{1^*\tau}\int_{\mathbb{R}^N}|u|^{1^*-2}uv \, \mbox{d}x,
\end{eqnarray*}
\begin{eqnarray*}\label{def:F'-t}
  \langle \widetilde{\mathcal{F}}_{\tau}'((u,\tau)),s\rangle &=& sNe^{-N\tau}\int_{\mathbb{R}^N}\Big(f(e^{N\tau}u)e^{N\tau}u-F(e^{N\tau}u)\Big) \, \mbox{d}x
  +se^{1^*\tau}\int_{\mathbb{R}^N}|u|^{1^*} \, \mbox{d}x
\end{eqnarray*}
for any $(v,s)\in E$. It is easy to see that
\begin{eqnarray*}
  \langle \widetilde{\mathcal{F}}'((u,\tau)),(v,s)\rangle &=& 
  \langle \widetilde{\mathcal{F}}_u'((u,\tau)),v\rangle + \langle \widetilde{\mathcal{F}}_{\tau}'(u,\tau),s\rangle.
\end{eqnarray*}
We define the set
\[
\widetilde{T}_{(u,\tau)}=\left\{z=(z_1,z_2)\in E: \langle\widetilde{\mathcal{F}}_u'((u,\tau)),z_1\rangle=0\right\}.
\]

The following lemma establishes the existence of approximate critical points and the related variational inequality.
Since $S_m$ is not $C^1$, the approach in \cite{Jean97} is not applicable here.
Recall that $\widetilde{J}_1((u,\tau))=e^{\tau}\int_{\mathbb{R}^N}|Du|$.

\begin{lem}\label{lem:4-3}
Suppose ($f_1$)-($f_4$) hold.
Let $0<\varepsilon<\widetilde{\gamma}_m-\widetilde{\gamma}_0$ and $g=(g_1,g_2)\in\widetilde{\Gamma}_m$ be such that
\[
\max_{t\in[0,1]}\widetilde{I}(g(t))\leq\widetilde{\gamma}_m+\varepsilon.
\]
Then there exists $(u,\tau)\in S_m\times\mathbb{R}$ such that
\begin{itemize}
  \item [(1)] $\widetilde{I}((u,\tau))\in[\widetilde{\gamma}_m-\varepsilon,\widetilde{\gamma}_m+\varepsilon]$;
  \item [(2)] $\min_{t\in[0,1]}\|(u,\tau)-g(t)\|_E\leq\sqrt{\varepsilon}$;
  \item [(3)] {$\forall z=(z_1,z_2)\in \widetilde{T}_{(u,\tau)}$} with $\|z\|_E=1$ and $\forall\kappa>0$ small enough, $$\widetilde{J}_1((u+\kappa z_1,\tau))-\widetilde{J}_1((u,\tau))+\kappa z_2\widetilde{J}_1((u,\tau))
-\langle \widetilde{\mathcal{F}}'((u,\tau)),\kappa z\rangle\geq-4\kappa\sqrt{\varepsilon}.$$
\end{itemize}
\end{lem}
\begin{proof}
  Like in \cite{Jean97}, we define $\Phi:\widetilde{\Gamma}_m\to\mathbb{R}$ by
\[
\Phi(h)=\max_{t\in[0,1]}\widetilde{I}(h(t))
\]
and endow on $\widetilde{\Gamma}_m$ with the metric
\[
d(\xi,\eta)=\max_{t\in[0,1]}\|\xi(t)-\eta(t)\|_E,\quad\forall \xi, \eta\in \widetilde{\Gamma}_m.
\]
Note that $\Phi$ is continuous and $\inf_{h\in\widetilde{\Gamma}_m}\Phi(h)=\widetilde{\gamma}_m>-\infty$, we can apply the Ekeland's variational principle for $\Phi$ on $\widetilde{\Gamma}_m$ (\cite[Theorem 4.1]{MaWi89}) and obtain that there exists $\eta=(\eta_1,\eta_2)\in\widetilde{\Gamma}_m$ such that

(i) $\Phi(\eta)\leq\Phi(g)$;

(ii) $d(\eta, g)\leq\sqrt{\varepsilon}$;

(iii) $\Phi(h)>\Phi(\eta)-\sqrt{\varepsilon}d(h,\eta)$.

We prove that there exists a $t\in[0,1]$ such that
\[
\widetilde{\gamma}_m-\varepsilon\leq\widetilde{I}(\eta(t)),
\]
and
\begin{eqnarray}\label{lem:4-3-1}
\widetilde{J}_1((\eta_1(t)+\kappa z_1,\eta_2(t)))-\widetilde{J}_1(\eta(t))+\kappa z_2\widetilde{J}_1(\eta(t))
-\langle \widetilde{\mathcal{F}}'(\eta(t)),\kappa z\rangle\geq-4\kappa\sqrt{\varepsilon}
\end{eqnarray}
for {$\forall z\in \widetilde{T}_{\eta(t)}$} with $\|z\|_E=1$ and $\forall\kappa>0$ small enough.
By contradiction, assume that for each $t\in S$, where
\[
S=\{t\in[0,1]:\widetilde{\gamma}_m-\varepsilon\leq\widetilde{I}(\eta(t))\},
\]
there exist a {$z(t)\in \widetilde{T}_{\eta(t)}$} with $\|z(t)\|_E=1$, a $\kappa(t)>0$ small such that
\begin{eqnarray}\label{lem:4-3-2}
\widetilde{J}_1((\eta_1(t)+\kappa(t) z_1(t),\eta_2(t)))-\widetilde{J}_1(\eta(t))+\kappa(t) z_2(t)\widetilde{J}_1(\eta(t))
-\langle \widetilde{\mathcal{F}}'(\eta(t)),\kappa(t) z(t)\rangle<-4\kappa(t)\sqrt{\varepsilon}. \nonumber\\
\end{eqnarray}
We divide the proof into three steps to obtain a contradiction.

{\bf Step 1.} We construct a continuous map $W: S\to E$  such that $\forall t\in S$, $0<\|W(t)\|_E\leq1$ and
\[
\widetilde{J}_1((\eta_1(t)+\kappa W_1(t),\eta_2(t)))-\widetilde{J}_1(\eta(t))+\kappa(t) W_2(t)\widetilde{J}_1(\eta(t)) \nonumber\\
-\langle \widetilde{\mathcal{F}}'(\eta(t)),\kappa W(t)\rangle<-4\kappa\sqrt{\varepsilon}
\]
for $\kappa>0$ small enough.

In fact, for each $t\in S$, since $\widetilde{J}_1$, $\widetilde{F}'$, $\eta$ are continuous and (\ref{lem:4-3-2}) holds, there exists a $\delta(t)>0$ such that $\forall s\in [0,1]$ with $|s-t|<\delta(t)$, there holds
\begin{eqnarray}\label{lem:4-3-3}
\widetilde{J}_1((\eta_1(s)+\kappa(t) z_1(t),\eta_2(s)))-\widetilde{J}_1(\eta(s))+\kappa(t) z_2(t)\widetilde{J}_1(\eta(s)) \nonumber\\
-\langle \widetilde{\mathcal{F}}'(\eta(s)),\kappa(t) z(t)\rangle<-4\kappa(t)\sqrt{\varepsilon}.
\end{eqnarray}
For later use, we may assume that $\delta(t)$ is so small that for $|s-t|<\delta(t)$, 
\begin{eqnarray}\label{lem:4-3-3+}
\langle \widetilde{\mathcal{F}}'(\eta(s)), z(t)\rangle\leq
\|\widetilde{\mathcal{F}}'(\eta(s))-\widetilde{\mathcal{F}}'(\eta(t))\|_{*}<\varepsilon.
\end{eqnarray}
For $t\in S$, we denote $B_{\delta(t)}(t)=\{s\in [0,1]: |s-t|<\delta(t)\}$.
Clearly, the family $\{B_{\delta(t)}(t):t\in S\}$ covers $S$.
Since $S$ is compact, the above covering admits a finite refinement $\{B_{\delta(t_i)}(t_i):t_i\in S\}_{i\in\Lambda}$. For each $i\in\Lambda$, let $\rho_i(s)$ be the distance from $s$ to the complement of $B_{\delta(t_i)}(t_i)$, then $\rho_i$ is Lipschitz continuous and vanishes outside $B_{\delta(t_i)}(t_i)$. Set 
\[
m_i(s)=\frac{\rho_i(s)}{\sum_{j\in\Lambda}\rho_j(s)}
\]
and let
\[
W(s)=\sum_{i\in\Lambda}m_i(s)z(t_i).
\]
We have $0<\|W(s)\|_E\leq1$, $\forall s\in S$; moreover, by {the convexity of $\widetilde{J}_1$ with respect to the first variable} and (\ref{lem:4-3-3}),
\begin{eqnarray*}\label{lem:4-3-4}
&&\widetilde{J}_1((\eta_1(s)+\kappa W_1(s),\eta_2(s)))-\widetilde{J}_1(\eta(s))+\kappa W_2(t)\widetilde{J}_1(\eta(s))
-\langle \widetilde{\mathcal{F}}'(\eta(s)),\kappa W(s)\rangle\\
&\leq&\sum_{i\in\Lambda}m_i(s)\left\{\widetilde{J}_1((\eta_1(s)+\kappa z_1(t_i),\eta_2(s)))-\widetilde{J}_1(\eta(s))+\kappa z_2(t_i)\widetilde{J}_1(\eta(s))
-\langle \widetilde{\mathcal{F}}'(\eta(s)),\kappa z(t_i)\rangle\right\}\\
&\leq&\sum_{i\in\Lambda}m_i(s)\left\{\left(1-\frac{\kappa}{\kappa(t_i)}\right)\widetilde{J}_1(\eta(s))
+\frac{\kappa}{\kappa(t_i)}\widetilde{J}_1((\eta_1(s)+\kappa(t_i)z_1(t_i),\eta_2(s)))\right. \\
&&\left.\quad -\widetilde{J}_1(\eta(s))+\kappa z_2(t_i)\widetilde{J}_1(\eta(s))
-\frac{\kappa}{\kappa(t_i)}\langle \widetilde{\mathcal{F}}'(\eta(s)),\kappa(t_i) z(t_i)\rangle\right\}\\
&=&\sum_{i\in\Lambda}m_i(s)\frac{\kappa}{\kappa(t_i)}\bigg\{\widetilde{J}_1((\eta_1(s)+\kappa(t_i)z_1(t_i),\eta_2(s)))-\widetilde{J}_1(\eta(s)) \\
&&\quad +\kappa(t_i) z_2(t_i)\widetilde{J}_1(\eta(s))-\langle \widetilde{\mathcal{F}}'(\eta(s)),\kappa(t_i) z(t_i)\rangle\bigg\}\\
\\&<&-4\kappa\sqrt{\varepsilon}
\end{eqnarray*}
provided that $\kappa<\min\{\kappa(t_i):i\in\Lambda\}$.

Since we work in $BV(\mathbb{R}^N)$ rather than $W^{1,1}(\mathbb{R}^N)$, we need to prove the continuity of $W$. 
Let $s_0\in S$ and let $\Lambda(s_0)$ be the set of $i\in\Lambda$ such that $B_{\delta(t_i)}(t_i)$ contains $s_0$. Since $\Lambda$ is finite, there exists a $\delta(s_0)>0$, depending on $s_0$, such that the set 
$$
B_{\delta(s_0)}(s_0):=\{s\in S: |s-s_0|<\delta(s_0)\}
$$
satisfying $B_{\delta(s_0)}(s_0)\subset B_{\delta(t_i)}(t_i)$ for any $i\in\Lambda(s_0)$. 
If $s_0\in S$ is an isolated point, then nothing can be said. Otherwise, for any $s\in B_{\delta(s_0)}(s_0)$, we have
\begin{eqnarray*}
  &&\|W(s)-W(s_0)\|_{BV} =\Big\|\sum_{i\in\Lambda}m_i(s)z(t_i)-\sum_{j\in\Lambda}m_j(s_0)z(t_j)\Big\|_{BV}\\
&=&\int_{\mathbb{R}^N}\Big|\sum_{i\in\Lambda}m_i(s)Dz(t_i)-\sum_{j\in\Lambda(s_0)}m_j(s_0)Dz(t_j)\Big|
+\int_{\mathbb{R}^N}\Big|\sum_{i\in\Lambda}m_i(s)z(t_i)-\sum_{j\in\Lambda(s_0)}m_j(s_0)z(t_j)\Big|\\
&\leq&\int_{\mathbb{R}^N}\sum_{i\in\Lambda(s_0)}|m_i(s)-m_i(s_0)||Dz(t_i)|
+\int_{\mathbb{R}^N}\sum_{i\in\Lambda\setminus\Lambda(s_0)}m_i(s)|Dz(t_i)|\\
&&+\int_{\mathbb{R}^N}\sum_{i\in\Lambda(s_0)}|m_i(s)-m_i(s_0)||z(t_i)|
+\int_{\mathbb{R}^N}\sum_{i\in\Lambda\setminus\Lambda(s_0)}m_i(s)|z(t_i)|\\
&\to&0
\end{eqnarray*}
as $s\to s_0$. This implies that $W(s)$ is continuous with respect to $s\in S$.

{\bf Step 2.} For $r>0$ small, we construct a path $g_r\in \widetilde{\Gamma}_m$ joining $(u_1,0)$ and $(u_2,0)$ and satisfying
\begin{eqnarray*}\label{lem:4-3-5}
d(g_r,\eta)\leq 2r.
\end{eqnarray*}
To this end, as in \cite{Jean97}, we define a continuous function $\Psi:[0,1]\to[0,1]$ by
\[
\Psi(t)=\left\{
          \begin{array}{ll}
            1, & \hbox{if $\widetilde{\gamma}_m\leq\widetilde{I}(\eta(t))$;} \\
            0, & \hbox{if $\widetilde{I}(\eta(t))\leq\widetilde{\gamma}_m-\varepsilon$.}
          \end{array}
        \right.
\]
For $t\in [0,1]$ and for $W(t)=(W_1(t),W_2(t))$ obtained in Step 1, we define, for $0\leq r\leq1/2$,
\[
\xi_1(t)=\eta_1(t)+r(1-e^{-r})\Psi(t)W_1(t).
\]
By definition, for $t\in[0,1]\setminus S$, $\xi_1(t)=\eta_1(t)$; moreover, 
\[
\|\xi_1(t)\|_1\geq\|\eta_1(t)\|_1-r(1-e^{-r})\Psi(t)\|W_1(t)\|_1\geq\frac{m}{2}
\]
for all $t\in S$ and $r\geq0$ small enough, so $\xi_1(t)\neq0$.
More precisely, we have
\begin{eqnarray}\label{lem:4-3-6}
\left| m-\|\xi_1(t)\|_1\right|\leq r(1-e^{-r})\Psi(t)\|W_1(t)\|_1.
\end{eqnarray}
We define the family of continuous mappings $g_r:[0,1]\to S_m\times\mathbb{R}$ with
\[
g_r(t)=
\left\{
  \begin{array}{ll}
    \left(m\|\xi_1(t)\|_1^{-1}\xi_1(t), \eta_2(t)+r\Psi(t)m\|\xi_1(t)\|_1^{-1}W_2(t)\right), & \hbox{if $t\in S$;} \\
    \eta(t), & \hbox{if $t\in[0,1]\setminus S$.}
  \end{array}
\right.
\]
Since $\eta\in\widetilde{\Gamma}_m$, for $t=0$ and $t=1$, we have
\[
\widetilde{I}(\eta(0))=\widetilde{I}((u_1,0))<\widetilde{\gamma}_m-\varepsilon
\]
and
\[
\widetilde{I}(\eta(1))=\widetilde{I}((u_2,0))<\widetilde{\gamma}_m-\varepsilon
\]
respectively according to \eqref{lem:4-2-3}, hence $\Psi(0)=\Psi(1)=0$ and so we have $\xi_1(0)=u_1$ and $\xi_1(1)=u_2$, this implies that $g_r\in \widetilde{\Gamma}_m$.
Now for $r>0$ small enough and $\forall t\in S$, by (\ref{lem:4-3-6}), we have
\begin{eqnarray*}\label{lem:4-3-7}
&&\|(g_r)_1(t)-\eta_1(t)\|_{BV}\\
&=&\|\left( m-\|\xi_1(t)\|_1\right)\|\xi_1(t)\|_1^{-1}\eta_1(t)
+r\Psi(t)m\|\xi_1(t)\|_1^{-1}W_1(t)\|_{BV}\\
&\leq&r(1-e^{-r})\Psi(t)\|W_1(t)\|_1\|\xi_1(t)\|_1^{-1}\|\eta_1(t)\|_{BV}+r(1-e^{-r})\Psi(t)m\|\xi_1(t)\|_1^{-1}\|W_1(t)\|_{BV}\\
&\leq&2r\Psi(t)\|W_1(t)\|_{BV}
\end{eqnarray*}
since $\sup_{t\in[0,1]}\|\eta_1(t)\|_{BV}<+\infty$. As for the second term in the norm of $E$, we have
\[
|(g_r)_2(t)-\eta_2(t)|_{\mathbb{R}}=r\Psi(t)m\|\xi_1(t)\|_1^{-1}|W_2(t)|_{\mathbb{R}}\leq 2r\Psi(t)|W_2(t)|_{\mathbb{R}}.
\]
Thus $\|g_r(t)-\eta(t)\|_{E}\leq 2r\Psi(t)\|W(t)\|_{E}$ and therefore,
\begin{eqnarray}\label{lem:4-3-8}
  d(g_r,\eta) &\leq& 2r.
\end{eqnarray}

{\bf Step 3.} By the continuity of $\widetilde{I}$ and the compactness of $[0,1]$, we derive from Step 1 that there exists $\widetilde{\delta}>0$ such that for each $t\in S$, $\forall (u,\tau)\in E$ with $\|(u,\tau)-\eta(t)\|_{E}\leq\widetilde{\delta}$, 
\[
\widetilde{J}_1((u+\kappa W_1(t),\eta_2(t)))-\widetilde{J}_1(\eta(t))+\kappa W_2(t)\widetilde{J}_1(\eta(t))
-\langle \widetilde{\mathcal{F}}'(\eta(t)),\kappa W(t)\rangle<-4\kappa\sqrt{\varepsilon}
\]
for $\kappa>0$ small.
Now for any $t\in S$, from Step 2, for $r>0$ sufficiently small, we have 
\begin{eqnarray}\label{lem:4-3-9}
&&\widetilde{J}_1((\eta_1(t)+r(1-e^{-r})\Psi(t)m\|\xi_1(t)\|_1^{-1}W_1(t),\eta_2(t)))-\widetilde{J}_1(\eta(t))
\nonumber\\
&&\quad +r\Psi(t)m\|\xi_1(t)\|_1^{-1}W_2(t)\widetilde{J}_1(\eta(t))-\langle \widetilde{\mathcal{F}}'(\eta(t)),r\Psi(t)m\|\xi_1(t)\|_1^{-1}W(t)\rangle \nonumber\\
&<&-4r\Psi(t)m\|\xi_1(t)\|_1^{-1}\sqrt{\varepsilon}.
\end{eqnarray}
Let us temporarily denote $k=\Psi(t)m\|\xi_1(t)\|_1^{-1}$ for simplicity.
Firstly, for $r>0$ sufficiently small,
\begin{eqnarray}\label{lem:4-3-10}
  && \widetilde{J}_1(g_r(t))-\widetilde{J}_1(\eta(t))
  -\left\{\widetilde{J}_1((\eta_1(t)+r(1-e^{-r})kW_1(t),\eta_2(t)))-\widetilde{J}_1(\eta(t))+rkW_2(t)\widetilde{J}_1(\eta(t))\right\} \nonumber\\
   &=& e^{\eta_2(t)+rkW_2(t)}\int_{\mathbb{R}^N}|D\eta_1(t)+r(1-e^{-r})kDW_1(t)|-e^{\eta_2(t)}\int_{\mathbb{R}^N}|D\eta_1(t)| \nonumber\\
   && -\bigg\{e^{\eta_2(t)}\int_{\mathbb{R}^N}|D\eta_1(t)+r(1-e^{-r})kDW_1(t)|-e^{\eta_2(t)}\int_{\mathbb{R}^N}|D\eta_1(t)| \nonumber \\
   &&\quad +rkW_2(t)e^{\eta_2(t)}\int_{\mathbb{R}^N}|D\eta_1(t)|\bigg\} \nonumber\\
   &=&(e^{rkW_2(t)}-1)e^{\eta_2(t)}\int_{\mathbb{R}^N}|D\eta_1(t)+r(1-e^{-r})kDW_1(t)|-rkW_2(t)e^{\eta_2(t)}\int_{\mathbb{R}^N}|D\eta_1(t)| \nonumber\\
   &=&({rkW_2(t)}+o(r))e^{\eta_2(t)}\int_{\mathbb{R}^N}|D\eta_1(t)+r(1-e^{-r})kDW_1(t)|-rkW_2(t)e^{\eta_2(t)}\int_{\mathbb{R}^N}|D\eta_1(t)| \nonumber\\
   &\leq&({rkW_2(t)}+o(r))e^{\eta_2(t)}r(1-e^{-r})k\int_{\mathbb{R}^N}|DW_1(t)| \nonumber\\
   &=& o(r)rk.
\end{eqnarray}
Secondly, since $\mathcal{F}$ is $C^1$, by the mean value theorem,
\begin{eqnarray}\label{lem:4-3-11}
&& \widetilde{\mathcal{F}}(g_r(t))-\widetilde{\mathcal{F}}(\eta(t))-\langle \widetilde{\mathcal{F}}'(\eta(t)),rkW(t)\rangle \nonumber\\
&=& \widetilde{\mathcal{F}}(g_r(t))-\widetilde{\mathcal{F}}(\eta(t))-\langle \widetilde{\mathcal{F}}'(\eta(t)),(r(1-e^{-r})kW_1(t),rkW_2(t))\rangle 
-\langle \widetilde{\mathcal{F}}'(\eta(t)),(re^{-r}kW_1(t),0)\rangle \nonumber\\
&=& \langle \widetilde{\mathcal{F}}'(\eta(t)+\theta(g_r(t)-\eta(t))),g_r(t)-\eta(t)\rangle
-\langle \widetilde{\mathcal{F}}'(\eta(t)),g_r(t)-\eta(t)\rangle
-\langle \widetilde{\mathcal{F}}'(\eta(t)),(re^{-r}kW_1(t),0)\rangle \nonumber\\
&=& o(r)rk-\langle \widetilde{\mathcal{F}}'(\eta(t)),(re^{-r}kW_1(t),0)\rangle.
\end{eqnarray}
Since $W\in \widetilde{T}_{\eta(t)}$ and $\mathcal{F}$ is $C^1$, the last term in the above equality 
\begin{eqnarray}\label{lem:4-3-12}
  && \langle\widetilde{\mathcal{F}}'(\eta(s)),(re^{-r}kW_1(t),0)\rangle \nonumber\\
  &=& re^{-r}k\sum_{i\in\Lambda}\langle\widetilde{\mathcal{F}}'(\eta(s)),(m_i(s)z_1(t_i),0)\rangle \nonumber\\
  &=& re^{-r}k\sum_{i\in\Lambda(s)}\langle\widetilde{\mathcal{F}}'(\eta(s))-\widetilde{\mathcal{F}}'(\eta(t_i)),(m_i(s)z_1(t_i),0)\rangle \nonumber\\
  &=& o(r)rk
\end{eqnarray}
according to (\ref{lem:4-3-3+}), where $\Lambda(s)$ is the set of $i\in\Lambda$ such that $B_{\delta(t_i)}(t_i)$ contains $s$.
Combining (\ref{lem:4-3-10}), (\ref{lem:4-3-11}) and (\ref{lem:4-3-12}), we conclude that
\begin{eqnarray}\label{lem:4-3-13}
\widetilde{I}(g_r(t))-\widetilde{I}(\eta(t))=\mbox{left-hand~side~of~(\ref{lem:4-3-9})}+o(r)rk.
\end{eqnarray}
Moreover, if $t\notin S$, then $\Psi(t)=0$ and we have $\widetilde{I}(g_r(t))=\widetilde{I}(\eta(t))$. 
Now fix $r>0$ small enough and assume that $t_0\in[0,1]$ is such that $\widetilde{I}(g_r(t_0))=\Phi(g_r)$. 
By \eqref{lem:4-3-9}, \eqref{lem:4-3-13} and the fact that $g_r\in\widetilde{\Gamma}_m$, we have
\[
\widetilde{I}(\eta(t_0))\geq\widetilde{I}(g_r(t_0))\geq\widetilde{\gamma}_m.
\]
This means that $t_0\in S$ and $\Psi(t_0)=1$. Thus
\[
\widetilde{I}(g_r(t_0))-\widetilde{I}(\eta(t_0))\leq-3r\sqrt{\varepsilon}.
\]
Consequently,
\[
\Phi(g_r)+3r\sqrt{\varepsilon}\leq\Phi(\eta).
\]
However, by property (iii) of Ekeland's variational principle, we have
\[
\Phi(g_r)+2r\sqrt{\varepsilon}\geq\Phi(g_r)+\sqrt{\varepsilon}d(g_r,\eta)>\Phi(\eta).
\]
This gives a contradiction, and hence \eqref{lem:4-3-1} holds. The proof of the lemma is completed.
\end{proof}

From Lemma \ref{lem:4-3}, we can derive a Palais-Smale sequence from a minimizing sequence.

\begin{cor}
  Suppose ($f_1$)-($f_4$) hold. Let $\{g_n\}\subset\widetilde{\Gamma}_m$ be such that
\[
\max_{t\in[0,1]}\widetilde{I}(g_n(t))\leq\widetilde{\gamma}_m+\frac{1}{n}.
\]
Then there exists a sequence $\{(u_n,\tau_n)\}\subset{S}_m\times\mathbb{R}$ such that
\begin{description}\label{cor:4-4}
  \item [(1)] $\widetilde{I}((u_n,\tau_n))\in[\widetilde{\gamma}_m-\frac{1}{n},\widetilde{\gamma}_m+\frac{1}{n}]$;
  \item [(2)] $\min_{t\in[0,1]}\|(u_n,\tau_n)-g_n(t)\|_E\leq{\frac{1}{\sqrt {n}}}$;
  \item [(3)] {$\forall z=(z_1,z_2)\in \widetilde{T}_{(u_n,\tau_n)}$} with $\|z\|_E=1$ and $\forall\kappa>0$ small enough, 
      \begin{eqnarray}\label{cor:4-4-1}
        \widetilde{J}_1((u_n+\kappa z_1,\tau_n))-\widetilde{J}_1((u_n,\tau_n))+\kappa z_2\widetilde{J}_1((u_n,\tau_n))
-\langle \widetilde{\mathcal{F}}'((u_n,\tau_n)),\kappa z\rangle\geq-\kappa{\frac{4}{\sqrt {n}}}\nonumber\\
      \end{eqnarray}
with $g$ replaced by $g_n$ in Lemma \ref{lem:4-3}.
\end{description}
\end{cor}

Now we can employ Corollary \ref{cor:4-4} to construct a minimizing Palais-Smale sequence in $S_m$ for the restricted functional $I|_{S_m}$ and derive a related approximate variational inequality.
To do this, we define the set
\begin{eqnarray}\label{def:T_u}
{T}_{u}=\left\{v\in BV(\mathbb{R}^N): \langle{\mathcal{F}}'(u),v\rangle=0\right\},
\end{eqnarray}
where
\[
\langle{\mathcal{F}}'(u),v\rangle=\int_{\mathbb{R}^N}\Big(f(u)v+|u|^{1^*-2}uv\Big) \, \mbox{d}x.
\]

\begin{prop}\label{prop:4-5}
  Assume that ($f_1$)-($f_4$) hold. Then there exists a sequence $\{v_n\}\subset S_m$ which is bounded in $BV(\mathbb{R}^N)$ such that 
\begin{eqnarray}\label{prop:4-5-1}
I(v_n)\to\gamma_m
\end{eqnarray}
and
\begin{eqnarray}\label{prop:4-5-11
}
P(v_n)\to0,
\end{eqnarray}
and for any {$h\in {T}_{u_n}$} with $\|h\|_{BV}=1$, and any sufficiently small $\kappa>0$, 
\begin{eqnarray}\label{prop:4-5-2}
J_1(v_n+\kappa h)-J_1(v_n)-\langle\mathcal{F}'(v_n),\kappa h\rangle\geq-\kappa\frac{8}{\sqrt{n}}.
\end{eqnarray}
\end{prop}
\begin{proof}
By Lemma \ref{lem:4-2}, we have $\Gamma_m\subset\widetilde{\Gamma}_m$ and $\gamma_m=\widetilde{\gamma}_m$. 
Thus we may assume that there exists a minimizing sequence $\{g_n\}\subset\widetilde{\Gamma}_m$ of the form $g_n(t)=((g_n)_1(t),0)\in E$, $\forall t\in[0,1]$, such that $\Phi(g_n)\in[\widetilde{\gamma}_m-\frac{1}{n},\widetilde{\gamma}_m+\frac{1}{n}]$.
Then for each $n\in\mathbb{N}$, there exists $(u_n,\tau_n)\in S_m\times\mathbb{R}$ satisfying the properties presented in Corollary \ref{cor:4-4}.

For each $(u_n,\tau_n)$, if we choose $z=(0,1)$ and $z = (0,-1)$ respectively in (\ref{cor:4-4-1}), then by \eqref{def:tideF'}, we obtain
\begin{eqnarray}\label{prop:4-5-3}
  P(\tau_n*u_n)&=&e^{\tau_n}\int_{\mathbb{R}^N}|Du_n|
  -Ne^{-N\tau_n}\int_{\mathbb{R}^N}\Big(f(e^{N\tau_n}u_n)e^{N\tau_n}u_n-F(e^{N\tau_n}u_n)\Big) \, \mbox{d}x\nonumber\\
  &&-e^{1^*\tau_n}\int_{\mathbb{R}^N}|u_n|^{1^*} \, \mbox{d}x
  = \frac{4}{\sqrt{n}}.
\end{eqnarray} 
On the other hand, by (1) of Corollary \ref{cor:4-4},
\begin{eqnarray}\label{prop:4-5-4}
  \widetilde{I}((u_n,\tau_n)) = e^{\tau_n}\int_{\mathbb{R}^N}|Du_n|
  -e^{-N\tau_n}\int_{\mathbb{R}^N}F(e^{N\tau_n}u_n) \, \mbox{d}x-\frac{e^{1^*\tau_n}}{1^*}\int_{\mathbb{R}^N}|u_n|^{1^*} \, \mbox{d}x
  =\widetilde{\gamma}_m+o(1). \nonumber\\
\end{eqnarray}
Subtracting (\ref{prop:4-5-3}) from (\ref{prop:4-5-4}) yields 
\begin{eqnarray}\label{prop:4-5-4-I-P}
&&Ne^{-N\tau_n}\int_{\mathbb{R}^N}\Big(f(e^{N\tau_n}u_n)e^{N\tau_n}u_n-\big(1+\frac{1}{N}\big)F(e^{N\tau_n}u_n)\Big) \, \mbox{d}x \nonumber\\
&&\quad +\frac{1}{N}e^{1^*\tau_n}\int_{\mathbb{R}^N}|u_n|^{1^*} \, \mbox{d}x
\leq \widetilde{\gamma}_m+o(1)+\frac{4}{\sqrt{n}}.
\end{eqnarray} 
By \eqref{ineq:f-F-1}, this implies that the two integrals on the left-hand side of the above inequality are both bounded.
Using \eqref{lem:2-3-1} and the H\"{o}lder inequality, we conclude that $\{e^{-N\tau_n}\int_{\mathbb{R}^N}F(e^{N\tau_n}u_n)\mbox{d}x\}$ is bounded and then 
$\{e^{\tau_n}\int_{\mathbb{R}^N}|Du_n|\}$ is also bounded by (\ref{prop:4-5-4}).

Now let $v_n=\tau_n*u_n$. Clearly, $\{v_n\}\subset S_m$ and is bounded in $BV(\mathbb{R}^N)$;
moreover, 
\[
I(v_n)=I(\tau_n*u_n)=\widetilde{I}((u_n,\tau_n))\in [{\gamma}_m-\frac{1}{n},{\gamma}_m+\frac{1}{n}].
\]
For any $h\in BV(\mathbb{R}^N)$ with $\|h\|_{BV}=1$, using the change of variables $y=e^{\tau_n}x$ and through direct computation, we have
\begin{eqnarray}\label{prop:4-5-5}
  &&J_1(v_n+\kappa h)-J_1(v_n)-\langle\mathcal{F}'(v_n),\kappa h\rangle \nonumber\\
&=& e^{\tau_n}\int_{\mathbb{R}^N}|D_yu_n(y)+\kappa e^{-N\tau_n}D_yh(e^{-\tau_n}y)|-e^{\tau_n}\int_{\mathbb{R}^N}|D_yu_n(y)| \nonumber\\
&&-\kappa\int_{\mathbb{R}^N}f(e^{N\tau_n}u_n(y))e^{-N\tau_n}h(e^{-\tau_n}y) \, \mbox{d}y \nonumber\\
&&-\kappa\int_{\mathbb{R}^N}e^{(1^*-1)N\tau_n}|u_n(y)|^{1^*-2}u_n(y)e^{-N\tau_n}h(e^{-\tau_n}y) \, \mbox{d}y \nonumber\\
&=&\widetilde{J}_1((u_n,\tau_n)+\kappa(\hat{h},0))-\widetilde{J}_1((u_n,\tau_n))
-\langle\widetilde{\mathcal{F}}'((u_n,\tau_n)),\kappa(\hat{h},0)\rangle \nonumber\\
&\geq&-\kappa\frac{8}{\sqrt{n}}
\end{eqnarray}
according to Corollary \ref{cor:4-4}, where $\hat{h}(y)=e^{-N\tau_n}h(e^{-\tau_n}y)$. Here note that 
\begin{eqnarray*}
  \|(\hat{h},0)\|_{E} &=& \|\hat{h}\|_{BV}
=\int_{\mathbb{R}^N}|\hat{h}| \, \mbox{d}x+\int_{\mathbb{R}^N}|D\hat{h}|\\
&=&\int_{\mathbb{R}^N}|{h}| \, \mbox{d}x+e^{-\tau_n}\int_{\mathbb{R}^N}|{Dh}|\\
&\leq&2\|h\|_{BV}
\end{eqnarray*}
Since $e^{-\tau_n}\leq2$ for $n$ large enough, we obtain $\hat{h}\in BV(\mathbb{R}^N)$. 
This completes the proof.
\end{proof}

We remark that for arbitrary $h\in T_u$,
if we replace $h$ by $h/\|h\|_{BV}$ and $\kappa$ by $\kappa\|h\|_{BV}$ respectively, we easily deduce that \eqref{prop:4-5-5} still holds.
We retain $\langle\mathcal{F}'(v_n),\kappa h\rangle$ in \eqref{prop:4-5-2} to make the inequality look "more complete".
We next provide a useful formula derived from (\ref{prop:4-5-2}). 

\begin{cor}\label{cor:4-6}
  Assume that ($f_1$)-($f_4$) hold. Let $\{v_n\}\subset S_m$ be the bounded sequence of $I|_{S_m}$ obtained in Proposition \ref{prop:4-5}. Then $\forall\varphi\in C^1(\mathbb{R}^N)\cap L^{\infty}(\mathbb{R}^N)$, we have
\begin{eqnarray}\label{cor:4-6-1}
  \int_{\mathbb{R}^N}\varphi|Dv_n|-\frac{1}{k_n}\langle\mathcal{F}'(v_n),\varphi v_n\rangle\int_{\mathbb{R}^N}|Dv_n|
  +\int_{\mathbb{R}^N}|v_nD\varphi| \, \mbox{d}x \geq-\frac{8}{\sqrt{n}},
\end{eqnarray}  
where $k_n:=\langle\mathcal{F}'(v_n),v_n\rangle$.
\end{cor}
\begin{proof}
  For $\forall\varphi\in BV(\mathbb{R}^N)\cap L^{\infty}(\mathbb{R}^N)$, we set
  \[
  h=\varphi v_n-\frac{1}{k_n}\langle\mathcal{F}'(v_n),\varphi v_n\rangle v_n,
  \]
  then $h\in T_{v_n}$. Inserting $h$ into (\ref{prop:4-5-2}) we get
  \[
  \int_{\mathbb{R}^N}\left|\left(1-\frac{\kappa}{k_n}\langle\mathcal{F}'(v_n),\varphi v_n\rangle+\kappa\varphi\right)Dv_n\right|
  +\kappa\int_{\mathbb{R}^N}|v_nD\varphi| \, \mbox{d}x-\int_{\mathbb{R}^N}|Dv_n|\geq-\kappa\frac{8}{\sqrt{n}}.
  \]
  For sufficiently small $\kappa>0$, the coefficient of $Dv_n$ in the bracket of the first integral is positive, so we can extract it from the total variational.
  In this way, we derive \eqref{cor:4-6-1}.
\end{proof}

In the following, for a sequence $\{u_n\}\subset BV(\mathbb{R}^N)\setminus\{0\}$, we use the notation
\begin{eqnarray*}\label{def:eta-n}
  \eta_n(u_n) &=& \frac{\|Du_n\|_{1}}{\langle\mathcal{F}'(u_n),u_n\rangle}
=\frac{\int_{\mathbb{R}^N}|Du_n|}{\int_{\mathbb{R}^N}f(u_n)u_n \, \mbox{d}x+\int_{\mathbb{R}^N}|u_n|^{1^*} \, \mbox{d}x}
\end{eqnarray*}
and
\begin{eqnarray*}\label{def:eta}
  \eta(u_n) &=& \liminf_{n\to\infty}\eta_n(u_n).
\end{eqnarray*}

We next employ the Lions' concentration compactness lemma to prove the local compactness of the minimizing sequence in $L^{1^*}(\mathbb{R}^N)$.

\begin{prop}\label{prop:loc-compact}
  Assume that ($f_1$)-($f_4$) hold. 
  Let $\{v_n\}\subset S_m$ be the bounded sequence of $I|_{S_m}$  obtained in Proposition \ref{prop:4-5} at level $c\in(0,\frac{1}{N}S^N)$. Then there exist a subsequence (up to translations in $\mathbb{R}^N$, still denoted by $\{v_n\}$) and $v\in S_m$ such that $v_n\to v$ in $L^{1^*}_{loc}(\mathbb{R}^N)$, provided that $\eta(v_n)\leq1$.
\end{prop}
\begin{proof}
   Employing the Radon-Nikod\'{y}m Theorem and arguing as in the Lions' concentration compactness lemma \cite{Lions85}, we obtain that there exist two nonnegative measures $\mu,\nu\in\mathcal{M}(\mathbb{R}^N)$ such that $|Dv_n|\rightharpoonup\mu$ and $|v_n|^{1^*}\rightharpoonup\nu$ in $\mathcal{M}(\mathbb{R}^N)$; moreover, there exists an at most countable index set $J_{\nu}$ such that
\begin{eqnarray}\label{prop:loc-compact-1}
\nu=|v|^{1^*}+\sum_{j\in J_{\nu}}\nu_j\delta_{x_j},\quad\nu_j>0,
\end{eqnarray}
where $x_j$'s are singular points of the measure $\nu$.
  For an arbitrary $j\in J_{\nu}$, we define $\varphi_{\rho}\in C_0^\infty(\mathbb{R}^N)$ by
\begin{eqnarray*}
\varphi_\rho(x)=\left\{
             \begin{array}{ll}
               1, & \hbox{$|x-x_j|\leq\frac{\rho}{2}$;} \\
               0, & \hbox{$|x-x_j|\geq\rho$;} \\
               |\nabla\varphi_\rho|\leq\frac{2}{\rho}, & \hbox{else.}
             \end{array}
           \right.
\end{eqnarray*}
Since $\{v_n\}$ is bounded in $BV(\mathbb{R}^N)$, the sequence $\{\varphi_\rho v_n\}$ is also bounded in $BV(\mathbb{R}^N)$. 
Taking $\varphi=\varphi_\rho$ as a test function in \eqref{cor:4-6-1}, we obtain
\begin{eqnarray}\label{prop:loc-compact-2}
  \int_{\mathbb{R}^N}\varphi_\rho|Dv_n|-\frac{1}{k_n}\langle\mathcal{F}'(v_n),\varphi_\rho v_n\rangle\int_{\mathbb{R}^N}|Dv_n|
  +\int_{\mathbb{R}^N}|v_nD\varphi_\rho| \, \mbox{d}x \geq-\frac{8}{\sqrt{n}}.
\end{eqnarray}
Taking $\varphi=-\varphi_\rho$ as a test function in \eqref{cor:4-6-1}, we can similar get
\begin{eqnarray}\label{prop:loc-compact-3}
  -\int_{\mathbb{R}^N}\varphi_\rho|Dv_n|+\frac{1}{k_n}\langle\mathcal{F}'(v_n),\varphi_\rho v_n\rangle\int_{\mathbb{R}^N}|Dv_n|
  +\int_{\mathbb{R}^N}|v_nD\varphi_\rho| \, \mbox{d}x \geq-\frac{8}{\sqrt{n}}.
\end{eqnarray}
Combining \eqref{prop:loc-compact-2} and \eqref{prop:loc-compact-3}, we deduce that
\begin{eqnarray}\label{prop:loc-compact-4}
  \left|\int_{\mathbb{R}^N}\varphi_\rho|Dv_n|-{\eta_n(v_n)}\int_{\mathbb{R}^N}f(v_n)v_n\varphi_\rho \, \mbox{d}x-{\eta_n(v_n)}\int_{\mathbb{R}^N}|v_n|^{1^*}\varphi_\rho \, \mbox{d}x\right|
  \leq\int_{\mathbb{R}^N}|v_nD\varphi_\rho| \, \mbox{d}x +\frac{8}{\sqrt{n}}. \nonumber\\
\end{eqnarray}
For $1\leq p<1^*$, the embedding $BV(\mathbb{R}^N)\hookrightarrow L^p(\mathbb{R}^N)$ is locally compact. 
Thus, 
\[
\lim_{\rho\to0}\lim_{n\to\infty}\int_{\mathbb{R}^N}|v_nD\varphi_\rho| \, \mbox{d}x=0,\quad
\lim_{\rho\to0}\lim_{n\to\infty}\int_{\mathbb{R}^N}f(v_n)v_n\varphi_\rho \, \mbox{d}x=0.
\]
Therefore, (\ref{prop:loc-compact-4}) implies that
\begin{eqnarray}\label{prop:loc-compact-5}
  \left|\int_{\mathbb{R}^N}\varphi_\rho|Dv_n|-\eta_n(v_n)\int_{\mathbb{R}^N}|v_n|^{1^*}\varphi_\rho \, \mbox{d}x\right|\to0.
\end{eqnarray}
Employing (\ref{prop:loc-compact-1}), we derive from (\ref{prop:loc-compact-5}) that
\begin{eqnarray*}
  \eta(v_n)\nu_j&=&\lim_{\rho\to0}\eta(v_n)\left(\int_{\mathbb{R}^N}|v|^{1^*}\varphi_\rho \, \mbox{d}x+\sum_{j\in J_\nu}\int_{\mathbb{R}^N}\nu_j\delta_{x_j}\varphi_\rho \, \mbox{d}x\right)\\
&=&\lim_{\rho\to0}\eta(v_n)\int_{\mathbb{R}^N}\varphi_\rho \, \mbox{d}\nu
=\lim_{\rho\to0}\liminf_{n\to\infty}\eta_n(v_n)\int_{\mathbb{R}^N}|v_n|^{1^*}\varphi_\rho \, \mbox{d}x\\
&=&\lim_{\rho\to0}\liminf_{n\to\infty}\int_{\mathbb{R}^N}|Dv_n|\varphi_\rho
\geq\lim_{\rho\to0}\liminf_{n\to\infty}S\left(\int_{B_{\rho/2}(x_j)}|v_n|^{1^*} \, \mbox{d}x\right)^{1/1^*}\\
&\geq&S\nu_j^{1/1^*}.
\end{eqnarray*}
Thus $\nu_j=0$ or $\nu_j\geq (S/\eta(v_n))^N$. We claim that the latter is impossible. Indeed, note that $I(v_n)\to c\in(0,\frac{1}{N}S^N)$ and $P(v_n)\to0$ as $n\to\infty$,
by computing $I(v_n)-P(v_n)$ and using (\ref{ineq:f-F-1}), we derive that
\[
\frac{1}{N}S^N>c>\left(1-\frac{1}{1^*}\right)\int_{\mathbb{R}^N}\varphi_\rho \, \mbox{d}\nu
\geq\left(1-\frac{1}{1^*}\right)\nu_j\geq\frac{1}{N}\left(\frac{S}{\eta(v_n)}\right)^N\geq\frac{1}{N}S^N.
\]
Here the last inequality holds due to $\eta(v_n)\leq1$.
This contradiction proves that $\nu_j=0$, $\forall j\in J_\nu$, and hence $J_\nu=\emptyset$. Consequently, we obtain that
\[
\lim_{n\to\infty}\int_{\mathbb{R}^N}|v_n|^{1^*}\varphi \, \mbox{d}x=\int_{\mathbb{R}^N}|v|^{1^*}\varphi \, \mbox{d}x,\quad\forall\varphi\in C_0(\mathbb{R}^N).
\]
Note that for $N\geq2$, $1<1^*<+\infty$. 
Moreover, $v_n\to v$ weakly in $L^{1^*}(\mathbb{R}^N)$, then the above limit implies $v_n\to v$ strongly in $L_{loc}^{1^*}(\mathbb{R}^N)$.
\end{proof}

\begin{prop}\label{prop:4-7}
  Assume that ($f_1$)-($f_5$) hold. Let $\{v_n\}\subset S_m$ be the bounded sequence for $I|_{S_m}$  obtained in Proposition \ref{prop:4-5} at level $\gamma_m\in(0,\frac{1}{N}S^N)$. Then $k_n\leq1$ provided that $\gamma_m$ is sufficiently small, where $k_n=\langle\mathcal{F}'(v_n),v_n\rangle$.
\end{prop}
\begin{proof}
  Let $\{v_n\}\subset S_m$ be the sequence obtained in Proposition \ref{prop:4-5}. We have
\begin{eqnarray*}\label{prop:4-7-1}
I(v_n)=\int_{\mathbb{R}^N}|Dv_n|-\int_{\mathbb{R}^N}F(v_n) \, \mbox{d}x
-\frac{1}{1^*}\int_{\mathbb{R}^N}|v_n|^{1^*} \, \mbox{d}x=\gamma_m+o(1)
\end{eqnarray*}
and
\begin{eqnarray*}\label{prop:4-7-2}
P(v_n)=\int_{\mathbb{R}^N}|Dv_n|-N\int_{\mathbb{R}^N}\Big(f(v_n)v_n-F(v_n)\Big) \, \mbox{d}x
-\int_{\mathbb{R}^N}|v_n|^{1^*} \, \mbox{d}x=o_n(1).
\end{eqnarray*}
Here and in the following, $o_n(1)$ denotes a quantity that tends to zero as $n$ tends to infinity.
Note that by ($f_3$),
\begin{eqnarray*}\label{prop:4-7-3}
I(v_n)-P(v_n)&=&N\int_{\mathbb{R}^N}\Big(f(v_n(x))v_n(x)-\Big(1+\frac{1}{N}\Big)F(v_n(x))\Big)
+\frac{1}{N}\int_{\mathbb{R}^N}|v_n(x)|^{1^*}\\
&\geq&N\left(\theta-\Big(1+\frac{1}{N}\Big)\right)\int_{\mathbb{R}^N}F(v_n(x)) \, \mbox{d}x
+\frac{1}{N}\int_{\mathbb{R}^N}|v_n(x)|^{1^*} \, \mbox{d}x.
\end{eqnarray*}
This shows that if $\gamma_m$ is small enough, we have
\[
\int_{\mathbb{R}^N}|v_n(x)|^{1^*}\leq{N}\gamma_m+o(1)+o_n(1)\leq\frac{1}{2}
\]
and by ($f_5$),
\[
\int_{\mathbb{R}^N}f(v_n(x))v_n \, \mbox{d}x\leq1^*\int_{\mathbb{R}^N}F(v_n(x)) \, \mbox{d}x
\leq\frac{1^*}{N\theta-N-1}\gamma_m+o(1)+o_n(1)\leq\frac{1}{2}.
\]
This proves $k_n\leq1$.
\end{proof}

With the locally strong convergence in $L^{1^*}(\mathbb{R}^N)$, we are able to recover the compactness of the sequence in $L^{r}(\mathbb{R}^N)$, $1\leq r\leq 1^*$. We remark that the approximating variational inequality (\ref{cor:4-6-1}) is of great help. 

{
\begin{prop}\label{prop:4-8}
  Assume that ($f_1$)-($f_5$) hold. Suppose further that $\gamma_m\in(0,\frac{1}{N}S^N)$. Then the bounded sequence $\{v_n\}\subset S_m$ obtained in Proposition \ref{prop:4-5} is compact in $L^r(\mathbb{R}^N)$, $\forall r\in[1,1^*]$, provided that $\gamma_m$ is sufficiently small.
\end{prop}}
\begin{proof}
Let $\{v_n\}\subset S_m$ be the sequence obtained in Proposition \ref{prop:4-5}. 
We define a sequence of nonnegative functions $\rho_n: S_m\to\mathbb{R}$ by
\begin{eqnarray*}\label{prop:4-8-3}
  \rho_n(x) &=& |v_n(x)|+ N\Big(f(v_n(x))v_n(x)-\Big(1+\frac{1}{N}\Big)F(v_n(x))\Big)
+\frac{1}{N}|v_n(x)|^{1^*}.
\end{eqnarray*}
We have 
\begin{eqnarray}\label{prop:4-8-4}
\int_{\mathbb{R}^N}\rho_n(x) \, \mbox{d}x=J_2(v_n)+I(v_n)-P(v_n)=m+\gamma_m+o(1)+o_n(1),
\end{eqnarray} 
where $J_2(v_n)=\int_{\mathbb{R}^N}|v_n| \, \mbox{d}x$. 
Since $\{\rho_n\}$ is a nonnegative bounded sequence in $L^1(\mathbb{R}^N)$, according to the Concentration Compactness Principle, there exists a subsequence, still denoted by itself, such that exactly one of the following statements holds.

(i) Vanishing. For any fixed $R>0$, 
\[
\lim_{n\to\infty}\sup_{y\in\mathbb{R}^N}\int_{B_R(y)}\rho_n(x) \, \mbox{d}x\to0.
\]

(ii) Compactness. There exists $\{y_n\}\subset\mathbb{R}^N$ such that $\forall \varepsilon>0$, $\exists R>0$ such that
\[
\int_{B_R(y_n)}\rho_n(x) \, \mbox{d}x\geq m+\gamma_m-\varepsilon.
\]

(iii) Dichotomy. There exist $\{y_n\}\subset\mathbb{R}^N$, $\alpha\in(0,m+\gamma_m)$, $R>0$, $R_n\to+\infty$ such that the functions 
\[
\rho_n^{(1)}(x)=\xi^{(1)}\Big(\frac{x-y_n}{R}\Big)\rho_n(x)\quad\mbox{and}\quad
\rho_n^{(2)}(x)=\xi^{(2)}\Big(\frac{x-y_n}{R_n}\Big)\rho_n(x)
\]
satisfy
\[
\int_{\mathbb{R}^N}\rho_n^{(1)}(x) \, \mbox{d}x\to\alpha\quad\mbox{and}\quad
\int_{\mathbb{R}^N}\rho_n^{(2)}(x) \, \mbox{d}x\to m+\gamma_m-\alpha,
\]
where $\xi^{(1)}$, $\xi^{(2)}$ are smooth cutoff functions satisfying that $\xi^{(1)}(x)=1$, $\xi^{(2)}(x)=0$ for $|x|\leq1$, $\xi^{(1)}(x)=0$, $\xi^{(2)}(x)=1$ for $|x|\geq2$ and $0\leq\xi^{(1)}(x), \xi^{(2)}(x)\leq1$.

We show that only the Compactness case occurs by excluding the other two possibilities.
Clearly, the Vanishing case cannot occur because $J_2(v_n)= m$, $\forall n\in\mathbb{N}$; otherwise, we would have
$v_n\to0$ in $L^{r}(\mathbb{R}^N)$ for all $1\leq r\leq 1^*$ and then $J_2(v_n)\to0$, a contradiction.
In view of this, we only need to deal with the Dichotomy case.
Since the problem is translation invariant, we may assume that $y_n=0$, $\forall n\in\mathbb{N}$; 
otherwise, one needs to prove that $\{y_n\}$ is a bounded sequence.

Now assume that Dichotomy occurs. Intuitively, $\rho_n$ is divided into two parts. 
Let
\[
v_n^{(1)}(x)=\xi^{(1)}\Big(\frac{x}{R}\Big)v_n(x)\quad\mbox{and}\quad
v_n^{(2)}(x)=\xi^{(2)}\Big(\frac{x}{R_n}\Big)v_n(x)
\]
and define $II:=I-P$. 
Then there exits $R>0$ sufficiently large such that 
\begin{eqnarray}\label{prop:4-8-5}
&& |II(v_n)-II(v_n^{(1)})-II(v_n^{(2)})| \nonumber\\
&\leq& C\int_{R\leq|x|\leq 2R_n}\left\{N\Big(f(v_n(x))v_n(x)-\Big(1+\frac{1}{N}\Big)F(v_n(x))\Big)
+\frac{1}{N}|v_n(x)|^{1^*}\right\} \, \mbox{d}x \nonumber\\
&\leq& Co_n(1)
\end{eqnarray}
and similarly,
\begin{eqnarray}\label{prop:4-8-6}
|J_2(v_n)-J_2(v_n^{(1)})-J_2(v_n^{(2)})|\leq Co_n(1),
\end{eqnarray}
where $C>0$ is a constant.

Note that $P(v_n)=o_n(1)$ can be written in the following form
\begin{eqnarray*}\label{prop:4-8-9}
  P(v_n)=\int_{\mathbb{R}^N}|Dv_n|-(N-1)\int_{\mathbb{R}^N}f(v_n)v_n \, \mbox{d}x+N\int_{\mathbb{R}^N}F(v_n) \, \mbox{d}x-\langle{\mathcal{F}}'(v_n),v_n\rangle=o_n(1).
\end{eqnarray*}
Dividing the above equality by $k_n=\langle{\mathcal{F}}'(v_n),v_n\rangle$, we get
\begin{eqnarray}\label{prop:4-8-10}
  \frac{1}{k_n}\int_{\mathbb{R}^N}|Dv_n|-\frac{(N-1)}{k_n}\int_{\mathbb{R}^N}\Big(f(v_n)v_n-1^*F(v_n)\Big) \, \mbox{d}x-1=o_n(1)
\end{eqnarray}
or equivalently,
\begin{eqnarray*}\label{prop:4-8-11}
\eta_n(v_n)=1+\frac{(N-1)}{k_n}\int_{\mathbb{R}^N}\Big(f(v_n)v_n-1^*F(v_n)\Big) \, \mbox{d}x+o_n(1).
\end{eqnarray*}
Taking $\liminf$ on both sides of the above equality, we get $\eta(v_n)\leq1$ by ($f_5$). Then according to Proposition \ref{prop:loc-compact}, we have $v_n\to v$ in $L^{1^*}_{loc}(\mathbb{R}^N)$.

For simplicity, we denote $\xi^{{1}}_R(\cdot)=\xi^{{1}}(\cdot/R)$.
Note that $\langle{\mathcal{F}}'(v_n),v_n\rangle$ also dichotomizes. We may assume that there exists $\alpha_1\in[0,1]$ such that $\frac{1}{k_n}\langle{\mathcal{F}}'(v_n),\xi^{(1)}_Rv_n\rangle\to\alpha_1$ and $\frac{1}{k_n}\langle{\mathcal{F}}'(v_n),(1-\xi^{(1)}_R)v_n\rangle\to1-\alpha_1$.

Now we choose $\varphi=\xi^{(1)}_R$ as a test function in (\ref{cor:4-6-1}) and obtain
\begin{eqnarray*}\label{prop:4-8-12}
\int_{\mathbb{R}^N}\xi^{(1)}_R|Dv_n|\geq\frac{1}{k_n}\langle{\mathcal{F}}'(v_n),\xi^{(1)}_Rv_n\rangle\int_{\mathbb{R}^N}|Dv_n|-\frac{2}{R}\int_{\{R\leq|x|\leq2R\}}|v_n| \, \mbox{d}x -\frac{8}{\sqrt{n}}. 
\end{eqnarray*}
Similarly, if we choose $\varphi=-\xi^{(1)}_R$ in \eqref{cor:4-6-1}, we get a similar reverse inequality. Consequently, we obtain
\begin{eqnarray}\label{prop:4-8-13}
\int_{\mathbb{R}^N}\xi^{(1)}_R|Dv_n|=\frac{1}{k_n}\langle{\mathcal{F}}'(v_n),\xi^{(1)}_Rv_n\rangle\int_{\mathbb{R}^N}|Dv_n|+o_n(1). 
\end{eqnarray}
This implies that 
\begin{eqnarray}\label{prop:4-8-14}
\frac{1}{\|Dv_n\|_1}\int_{\mathbb{R}^N}\xi^{(1)}_R|Dv_n|=\frac{1}{k_n}\langle{\mathcal{F}}'(v_n),\xi^{(1)}_Rv_n\rangle+o_n(1)\to\alpha_1 
\end{eqnarray}
as $n$ tends to infinity. 
Replacing $\xi^{(1)}_R$ by $(1-\xi^{(1)}_R)$, we get a similar result
\begin{eqnarray}\label{prop:4-8-15}
\frac{1}{\|Dv_n\|_1}\int_{\mathbb{R}^N}(1-\xi^{(1)}_R)|Dv_n|=\frac{1}{k_n}\langle{\mathcal{F}}'(v_n),(1-\xi^{(1)}_R)v_n\rangle+o_n(1)\to 1-\alpha_1.
\end{eqnarray}
The above two equalities show that $\|Dv_n\|_1$ dichotomizes at the same ratio as $\langle{\mathcal{F}}'(v_n),v_n\rangle$.

We claim that $0<\alpha_1<1$. Indeed, by the embedding inequality, we have 
\[
\int_{\mathbb{R}^N}\xi^{(1)}_R|Dv_n|\geq\int_{B_R(0)}|Dv_n|\geq C\int_{B_R(0)}|v_n| \, \mbox{d}x.
\]
Thus $\alpha_1=0$ would imply $\alpha=0$, which is a contradiction.

Adding up (\ref{prop:4-8-10}) and (\ref{prop:4-8-13}), we get
\begin{eqnarray*}\label{prop:4-8-16}
&&\frac{1}{k_n}\int_{\mathbb{R}^N}|Dv_n|+\int_{\mathbb{R}^N}\xi^{(1)}_R|Dv_n|-\frac{1}{k_n}\langle{\mathcal{F}}'(v_n),\xi^{(1)}_Rv_n\rangle\int_{\mathbb{R}^N}|Dv_n| \nonumber\\
&&\quad -\frac{(N-1)}{k_n}\int_{\mathbb{R}^N}\Big(f(v_n)v_n-1^*F(v_n)\Big) \, \mbox{d}x-1=o_n(1).
\end{eqnarray*}
Then by ($f_5$), we have
\begin{eqnarray*}\label{prop:4-8-17}
0 &\leq& -\frac{(N-1)}{k_n}\int_{\mathbb{R}^N}\Big(f(v_n)v_n-1^*F(v_n)\Big) \, \mbox{d}x \nonumber\\
&\leq& 1-\int_{\mathbb{R}^N}\xi^{(1)}_R|Dv_n|
-\frac{1}{k_n}\langle{\mathcal{F}}'(v_n),(1-\xi^{(1)}_R)v_n\rangle\int_{\mathbb{R}^N}|Dv_n|
+\left(1-\frac{1}{k_n}\right)\int_{\mathbb{R}^N}|Dv_n|+o_n(1).
\end{eqnarray*}
Note that by Proposition \ref{prop:4-7}, we have $k_n\leq1$ for $\gamma_m>0$ sufficiently small.
Thus by \eqref{prop:4-8-14}-\eqref{prop:4-8-15} we obtain
\begin{eqnarray*}\label{prop:4-8-18}
 -\frac{(N-1)}{k_n}\int_{\mathbb{R}^N}\Big(f(v_n)v_n-1^*F(v_n)\Big) \, \mbox{d}x =0
\end{eqnarray*}
Furthermore, by Fatou's lemma,
\begin{eqnarray*}\label{prop:4-8-19}
0\leq -\frac{(N-1)}{k_n}\int_{\mathbb{R}^N}\Big(f(v_0)v_0-1^*F(v_0)\Big) \, \mbox{d}x
\leq -\frac{(N-1)}{k_n}\int_{\mathbb{R}^N}\Big(f(v_n)v_n-1^*F(v_n)\Big) \, \mbox{d}x =0.
\end{eqnarray*}
This implies that $v_0=0$, a contradiction. Thus the Dichotomy case cannot occur.
This completes the proof.
\end{proof}

\begin{rem}
For $a>0$, we define 
\begin{eqnarray*}\label{prop:4-8-7}
II_a=\inf\{II(v):v\in BV(\mathbb{R}^N), J_2(v)=a\}.
\end{eqnarray*}
Then according to (\ref{prop:4-8-5}), we have
\begin{eqnarray}\label{prop:4-8-8}
II_m+\varepsilon\geq II(v_n) &\geq& II(v_n^{(1)})+II(v_n^{(2)})-C\varepsilon \nonumber\\
&\geq& II_{m_n^{(1)}}+II_{m_n^{(2)}}-C\varepsilon \nonumber\\
&\to& {II_{m^{(1)}}+II_{m^{(2)}}-C\varepsilon,}
\end{eqnarray}
where 
\[
m_n^{(i)}:=J_2(v_n^{(i)}), \quad m^{(i)}:=\lim_{n\to\infty}m_n^{(i)},\quad i=1,2.
\]
It is obvious that $0\leq m^{(1)}, m^{(2)}\leq m$ and by (\ref{prop:4-8-6}), $m^{(1)}+m^{(2)}=m$.
We claim that $0< m^{(1)}, m^{(2)}< m$. In fact, if $m^{(1)}=0$, then $m^{(2)}=m$. Therefore, for $n$ sufficiently large,
\begin{eqnarray*}
  \alpha-\varepsilon < \int_{\mathbb{R}^N}\rho_n^{(1)}(x) \, \mbox{d}x\leq m_n^{(1)}+II(v_n^{(1)})+C\varepsilon.
\end{eqnarray*}
This shows that there exists $\delta>0$ such that
\[
II(v_n^{(1)})\geq\delta.
\]
But then from (\ref{prop:4-8-8}), we have
\[
II_m+\varepsilon\geq \delta+II_{m}-C\varepsilon.
\]
This gives a contradiction, and hence $m^{(1)}>0$; similarly, $m^{(2)}>0$. 
In conclusion, if Dichotomy occurs, then there exist $m^{(1)}, m^{(2)}\in(0,m)$ with $m^{(1)}+m^{(2)}=m$ such that
\[
II_{m}\geq II_{m^{(1)}}+II_{m^{(2)}}.
\]
For $f(u)=|u|^{q-2}u$, $1+\frac{1}{N}<q<1^*$, it is possible to adapt the methods in \cite{Soa201} to prove the reverse strict inequality
\[
II_{m}< II_{m^{(1)}}+II_{m^{(2)}}
\]
to obtain a contradiction and then recover the compactness of the sequence $\{v_n\}$.
For this special case, we can use the approach as in the proof of Theorem 1.1 of \cite{LLJ12} to show that $\gamma_m<\frac{1}{N}S^N$ for any $m>0$,
this means that we can drop the assumption "$m>0$ is sufficiently large" in Theorem \ref{thm:m3}.
However, for a general perturbation $f$, more monotonicity assumptions are needed.
\end{rem}

After recovering the compactness of the minimizing sequence, we can obtain a minimizer of the functional under the constraint. We show that this minimizer not only belongs to the constrained manifold, but also belongs to the Pohozaev manifold.

\begin{prop}\label{prop:4-9}
  Assume that ($f_1$)-($f_5$) hold and that $\gamma_m\in(0,\frac{1}{N}S^N)$ is sufficiently small. Then there exist
  a $v_0\in S_m$ with $P(v_0)=0$,
  a $z\in L^{\infty}(\mathbb{R}^N,\mathbb{R}^N)$ satisfying $\|z\|_\infty\leq1$, $(z,Dv_0)=|Dv_0|$, $-\mbox{div}z\in L_{\infty,N}(\mathbb{R}^N)$, 
  a $\gamma\in L^\infty(\mathbb{R}^N)$ with $\|\gamma\|_\infty\leq1$, $\gamma v_0=|v_0|$ a.e. in $\mathbb{R}^N$
  and
  a $\lambda>0$ such that $I(v_0)=\gamma_m$ and
\begin{eqnarray}\label{prop:4-9-1}
-\mbox{div}z+\lambda\gamma=f(v_0)+|v_0|^{1^*-2}v_0,\quad\mbox{in}~L_{\infty,N}(\mathbb{R}^N).
\end{eqnarray}
\end{prop}
\begin{proof}
By Proposition \ref{prop:4-5}, there exists a bounded sequence $\{v_n\}\subset BV(\mathbb{R}^N)$ such that $I(v_n)\to\gamma_m$ and $P(v_n)\to0$; 
moreover, by Proposition \ref{prop:4-8}, there exists a $v_0\in BV(\mathbb{R}^N)$, $v_0\neq0$ such that $v_n\to v_0$ strongly in $L^r(\mathbb{R}^N)$, $\forall 1\leq r\leq 1^*$.
It is obvious that $v_0\in S_m$ and by Fatou Lemma
\[
\liminf_{n\to\infty}\int_{\mathbb{R}^N}|Dv_n|\geq\int_{\mathbb{R}^N}|Dv_0|.
\]
Let $h=v_0-\frac{1}{k_n}\langle\mathcal{F}'(v_n),v_0\rangle v_n$, where $k_n=\langle\mathcal{F}'(v_n),v_n\rangle$, then we have $h\in T_{v_n}$. Inserting $h$ into (\ref{prop:4-5-2}), we get
\[
J_1\big(v_n+\kappa(v_0-\frac{1}{k_n}\langle\mathcal{F}'(v_n),v_0\rangle v_n)\big)-J_1(v_n)\geq-\kappa\frac{8}{\sqrt{n}},
\]
where $k_n=\langle\mathcal{F}'(v_n),v_0\rangle$. Then for $\kappa>0$ small enough, we can argue as for (\ref{prop:loc-compact-2}) to get
\[
\int_{\mathbb{R}^N}|Dv_0|\geq\frac{1}{k_n}\langle\mathcal{F}'(v_n),v_0\rangle\int_{\mathbb{R}^N}|Dv_n|-\frac{8}{\sqrt{n}}.
\]
Taking $\limsup$ on both sides of the above inequality and employing Proposition \ref{prop:4-8}, we obtain
\begin{eqnarray*}\label{prop:4-9-3}
\int_{\mathbb{R}^N}|Dv_0|\geq\limsup_{n\to\infty}\int_{\mathbb{R}^N}|Dv_n|.
\end{eqnarray*}
Consequently, we have 
\begin{eqnarray*}\label{prop:4-9-2}
\lim_{n\to\infty}\int_{\mathbb{R}^N}|Dv_n|=\int_{\mathbb{R}^N}|Dv_0|.
\end{eqnarray*}
This proves that $I(v_0)=\gamma_m$ and so $v_0$ achieves $\gamma_m$ in $S_m$. 
Moreover, 
\begin{eqnarray*}\label{prop:4-9-4}
P(v_0)=\lim_{n\to\infty}P(v_n)=0.
\end{eqnarray*}

Now since $v_0$ is a minimizer of $I$ restricted on $S_m$, according to \cite[Corollary 3.6]{DeSc98} (or cf. \cite[Theorem 9.6.1]{ABM06}), there exist $\lambda\in\mathbb{R}$ such that
\[
\partial{J_1(v_0)}+\lambda\partial{J_2(v_0)}-\partial{\mathcal{F}(v_0)}\ni0.
\]
More precisely, there exist $v_0^*\in\partial{J_1(v_0)}$, $\gamma_0^*\in\partial{J_2(v_0)}$ such that
\[
v_0^*+\lambda\gamma_0^*-\mathcal{F}'(v_0)=0.
\]
Arguing as the proof of Proposition 4.23 in \cite{KaSc07}, for the above $v_0^*$ and $\gamma_0^*$, there exist $z\in L^{\infty}(\mathbb{R}^N,\mathbb{R}^N)$  satisfying $\|z\|_\infty\leq1$, $(z,Dv_0)=|Dv_0|$, $-\mbox{div}z=v_0^*$ in $L_{\infty,N}(\mathbb{R}^N)$, 
and a $\gamma\in L^\infty(\mathbb{R}^N)$ with $\|\gamma\|_\infty\leq1$, $\gamma v_0=|v_0|$ a.e. in $\mathbb{R}^N$ such that
\begin{eqnarray}\label{prop:4-9-5}
-\mbox{div}z+\lambda\gamma-f(v_0)-|v_0|^{1^*-2}v_0=0.
\end{eqnarray}

To complete the proof, we show that $\lambda>0$.
Indeed, since $P(v_0)=0$, testing \eqref{prop:4-9-5} with $v_0$ gives
\begin{eqnarray*}
  \lambda &=& \frac{1}{m}\left[\int_{\mathbb{R}^N}f(v_0)v_0 \, \mbox{d}x+\int_{\mathbb{R}^N}|v_0|^{1^*} \, \mbox{d}x-\int_{\mathbb{R}^N}|Dv_0|\right]\\
&\geq&\frac{1}{m}\left[N\int_{\mathbb{R}^N}F(v_0) \, \mbox{d}x-(N-1)\int_{\mathbb{R}^N}f(v_0)v_0 \, \mbox{d}x-P(v_0)\right]>0
\end{eqnarray*}
by ($f_5$). This completes the proof.
\end{proof}

\begin{rem}
  For any $\varphi\in BV(\mathbb{R}^N)$, testing \eqref{prop:4-9-5} with $\varphi-v_0$ and using \eqref{Green}, we obtain
  \begin{eqnarray*}
    \int_{\mathbb{R}^N}\Big(f(v_0)+|v_0|^{1^*-2}v_0\Big)(\varphi-v_0) \, \mbox{d}x 
    &=& \int_{\mathbb{R}^N}\Big(-\mbox{div}z+\lambda\gamma\Big)(\varphi-v_0) \\
    &=&\int_{\mathbb{R}^N}\Big((z,D\varphi)-|Dv_0|\Big)+\lambda\int_{\mathbb{R}^N}(\gamma\varphi-|v_0|) \\
    &\leq&\int_{\mathbb{R}^N}(|D\varphi|-|Dv_0|)+\lambda\int_{\mathbb{R}^N}(|\varphi|-|v_0|).
  \end{eqnarray*}
\end{rem}

It remains to prove that the critical value satisfies $0<\gamma_m<\frac{1}{N}S^N$. 
We will give estimates in Section \ref{sec:Em} and then complete the proof of Theorem \ref{thm:m1}.

\section{Estimates of the critical values}\label{sec:Em}
\indent\indent
In this section, we follow the approach of \cite{JeLu20,LiZo24} to study properties of the critical value $E_m$.
We will prove that if ($f_1$)--($f_4$) hold, then the function $m\mapsto E_m$ is positive, continuous, and nonincreasing. 
If in addition, ($f_6$) holds, then
$
\lim_{m\to\infty}E_m=0^+.
$
This means that if we take $m$ sufficiently large so that $\gamma_m\leq E_m<\frac{1}{N}S^N$, then the bounded minimizing Palais-Smale sequence obtained in Section \ref{sec:exist} is compact.

\begin{lem}\label{lem:2-6}
  Suppose ($f_1$)-($f_4$) hold. Then for any $u\in BV(\mathbb{R}^N)\setminus\{0\}$, there exists a unique $\tau(u)\in\mathbb{R}$ such that $P(\tau(u)*u)=0$. Consequently, $I(\tau(u)*u)>I(\tau*u)$ for any $\tau\neq\tau(u)$. 
\end{lem}
\begin{proof}
  By \eqref{eq:I(tau-u)}, we have
  \begin{eqnarray*}
    \frac{d}{d\tau}I(\tau*u) &=& e^{\tau}\int_{\mathbb{R}^N}|Du|+Ne^{-N\tau}\int_{\mathbb{R}^N}F(e^{N\tau}u)dx\\
    &&-Ne^{-N\tau}\int_{\mathbb{R}^N}e^{N\tau}uf(e^{N\tau}u)dx-e^{1^*\tau}\int_{\mathbb{R}^N}|u|^{1^*}dx\\
    &=&\int_{\mathbb{R}^N}|D(\tau*u)|-N\int_{\mathbb{R}^N}\widetilde{F}(\tau*u)dx-\int_{\mathbb{R}^N}|\tau*u|^{1^*}dx\\
    &=&P(\tau*u).
  \end{eqnarray*}
  where $\widetilde{F}$ is defined in ($f_4$).
  By Lemma \ref{lem:2-5}, there exists $\tau(u)\in\mathbb{R}$ such that $I(\tau*u)$ achieves its global maximum at $\tau(u)$. Thus we have
  \[
  P(\tau(u)*u)=\frac{d}{d\tau}\Big|_{\tau=\tau(u)}I(\tau*u)=0.
  \]
  We prove the uniqueness of $\tau(u)$. Suppose, by contradiction, that there exist $\tau_1(u)<\tau_2(u)$ such that
  \[
  P(\tau_1(u)*u)=P(\tau_2(u)*u)=0.
  \]
  Note that 
  \begin{eqnarray*}
    P(\tau*u) &=& e^{\tau}\int_{\mathbb{R}^N}|Du|-Ne^{-N\tau}\int_{\mathbb{R}^N}\widetilde{F}(e^{N\tau}u)dx
    -e^{1^*\tau}\int_{\mathbb{R}^N}|u|^{1^*}dx\\
    &=&e^{\tau}\int_{\mathbb{R}^N}|Du|
    -Ne^{-N\tau}\int_{\mathbb{R}^N}\frac{\widetilde{F}(e^{N\tau}u)}{|e^{N\tau}u|^{1+\frac{1}{N}}}|e^{N\tau}u|^{1+\frac{1}{N}}dx
    -e^{1^*\tau}\int_{\mathbb{R}^N}|u|^{1^*}dx\\
    &=&e^{\tau}\left[\int_{\mathbb{R}^N}|Du|
    -N\int_{\mathbb{R}^N}\frac{\widetilde{F}(e^{N\tau}u)}{|e^{N\tau}u|^{1+\frac{1}{N}}}|u|^{1+\frac{1}{N}}dx
    -e^{(1^*-1)\tau}\int_{\mathbb{R}^N}|u|^{1^*}dx\right].
  \end{eqnarray*}
  Hence
  \begin{eqnarray}\label{lem:2-6-1}
    &&N\int_{\mathbb{R}^N}\frac{\widetilde{F}(e^{N\tau_1}u)}{|e^{N\tau_1}u|^{1+\frac{1}{N}}}|u|^{1+\frac{1}{N}}dx
    +e^{(1^*-1)\tau_1}\int_{\mathbb{R}^N}|u|^{1^*}dx \nonumber\\
    &=&N\int_{\mathbb{R}^N}\frac{\widetilde{F}(e^{N\tau_2}u)}{|e^{N\tau_2}u|^{1+\frac{1}{N}}}|u|^{1+\frac{1}{N}}dx
    +e^{(1^*-1)\tau_2}\int_{\mathbb{R}^N}|u|^{1^*}dx
  \end{eqnarray}
  Since by ($f_4$), $\tau\mapsto\frac{\widetilde{F}(e^{N\tau}u)}{|e^{N\tau}u|^{1+\frac{1}{N}}}$ is strictly increasing, equality (\ref{lem:2-6-1}) cannot hold. This implies that $\tau(u)$ is unique. 
  Now Lemma \ref{lem:2-5} shows that $\tau(u)$ is the unique maximum point of $I(\tau*u)$, and hence
  $I(\tau(u)*u)>I(\tau*u)$ for any $\tau\neq\tau(u)$.
\end{proof}

\begin{lem}\label{lem:2-7}
  Suppose ($f_1$)-($f_4$) hold. Then $E_m=\inf_{u\in\mathcal{P}_m}I(u)>0$ and $I$ is coercive on $\mathcal{P}_m$.
\end{lem}
\begin{proof}
  Clearly, $\mathcal{P}_m\neq\emptyset$ by Lemma \ref{lem:2-6}.
For sufficiently small $\delta>0$ and for any $u\in BV(\mathbb{R}^N)\setminus\{0\}$ with $\|u\|_1\leq m$, $\|Du\|_1\leq\delta$, arguing as in \eqref{lem:2-3-(1)}, we obtain 
  \begin{eqnarray}\label{cor:2-4}
  P(u)\geq\frac{1}{2}\|Du\|_1.
  \end{eqnarray}
  For any $u\in\mathcal{P}_m$, we have $P(u)=0$. This implies $\|Du\|_1>\delta$.
  Let $u\in\mathcal{P}_m$. Then $\tau(u)=0$. Taking $\tau_0=\ln\left(\frac{\delta}{\|Du\|_1}\right)$, we have $\|D(\tau_0*u))\|_1=e^{\tau_0}\|Du\|_1=\delta$. By Lemma \ref{lem:2-6}, (\ref{lem:2-3-(1)}), we obtain
  \[
  I(u)=I(0*u)\geq I(\tau_0*u)\geq\frac{1}{4}\|D(\tau_0*u)\|_1=\frac{\delta}{4}.
  \]
  Thus $E_m>0$.
  
  Now assume that there exist $\{u_n\}\subset\mathcal{P}_m$ and $d>0$ such that $\|u_n\|_{BV}\to+\infty$ and $I(u_n)\leq d$. Set $v_n=u_n/\|u_n\|_{BV}$, then up to a subsequence
  \begin{eqnarray*}
    &&v_n\rightharpoonup v \mbox{~in~}BV(\mathbb{R}^N), \\
    &&v_n\to v \mbox{~in~}L^r_{loc}(\mathbb{R}^N), 1\leq r<1^*, \\
    &&v_n\to v \mbox{~a.e.~in~}\mathbb{R}^N.
  \end{eqnarray*}
  Let $\rho>0$ and define
  \[
  \eta_n=\sup_{z\in\mathbb{R}^N}\int_{B_\rho(z)}|v_n|^{1^*}dx.
  \]
We claim that $\limsup_{n\to\infty}\eta_n>0$. Indeed, if not, then we can argue as in Lemma 1.21 of \cite{Will96} (see also Lemma 3.8 of \cite{BWW05}) to obtain that {$v_n\to0$ in $L^{1^*}(\mathbb{R}^N)$}, and consequently $\int_{\mathbb{R}^N}F(v_n)dx\to0$ as well. 
To see this, choose $\theta\in(0,1)$.
Then for any $y\in\mathbb{R}^N$, we have 
\begin{eqnarray*}
\int_{B_\rho(y)}|v_n|^{1^*}dx
&=&\left(\int_{B_\rho(y)}|v_n|^{1^*}dx\right)^{\theta}\left(\int_{B_\rho(y)}|v_n|^{1^*}dx\right)^{1-\theta}\\
&\leq&\left(\sup_{z\in\mathbb{R}^N}\int_{B_\rho(z)}|v_n|^{1^*}dx\right)^{\theta}\left(\int_{B_\rho(y)}|v_n|^{1^*}dx\right)^{1-\theta}.
\end{eqnarray*}
Cover $\mathbb{R}^N$ by balls of radius $\rho$ such that each point of $\mathbb{R}^N$ is contained in at most $N+1$ balls. As a consequence,
\[
\int_{\mathbb{R}^N}|v_n|^{1^*}dx
\leq(N+1)\left(\sup_{z\in\mathbb{R}^N}\int_{B_\rho(z)}|v_n|^{1^*}dx\right)^{\theta}\left(\int_{\mathbb{R}^N}|v_n|^{1^*}dx\right)^{1-\theta}\to0.
\]
That is, $v_n\to0$ in $L^{1^*}(\mathbb{R}^N)$.
For $k>0$, let $\tau_n=\ln(k/\alpha_n)$, where $\alpha_n=\|Du_n\|_1$. Then using (\ref{lem:2-3-1}) and H\"{o}lder inequality, we can derive that
  \begin{eqnarray}\label{lem:2-7-1}
    d &\geq& I(u_n)=\max_{\tau\geq0}I(\tau*u_n)\geq I(\tau_n*u_n) \nonumber\\
      &=& k\int_{\mathbb{R}^N}|Dv_n|-\int_{\mathbb{R}^N}F(\tau_n*u_n)dx-\frac{t^{1^*}}{1^*}\int_{\mathbb{R}^N}|v_n|^{1^*}dx \nonumber\\
      &=& k\int_{\mathbb{R}^N}|Dv_n|-o_n(1)\max\{k,k^{1^*}\}.
  \end{eqnarray}
  Since $\|v_n\|_1=\frac{m}{\|u_n\|_{BV}}\to0$ and $\|v_n\|_{BV}=\|v_n\|_1+\|Dv_n\|_1=1$, we have $\|Dv_n\|_1\to1$ and $\|Dv_n\|_1\geq\frac{1}{2}$ for large $n$. 
  Then we get a contradiction from \eqref{lem:2-7-1} if we choose $k>0$ sufficiently large and then letting $n\to\infty$. This proves the claim. As a result, we obtain
  \[
  0\leq\liminf_{n\to\infty}\frac{I(u_n)}{\|u_n\|_{BV}^{1^*}}\leq\liminf_{n\to\infty}\frac{1}{\|u_n\|_{BV}^{1^*-1}}-\limsup_{n\to\infty}\frac{1}{1^*}\int_{\mathbb{R}^N}|v_n|^{1^*}dx\leq -\limsup_{n\to\infty}\eta_n<0.
  \]
  This contradiction shows that $I$ is coercive on $\mathcal{P}_m$.
\end{proof}

\begin{lem}\label{lem:2-8}
  Suppose that ($f_1$)-($f_4$) hold, then the map $u\mapsto\tau(u)$ is continuous on $BV(\mathbb{R}^N)\setminus\{0\}$.
\end{lem}
\begin{proof}
It is obvious that the mapping $u\mapsto\tau(u)$ is well defined on $BV(\mathbb{R}^N)\setminus\{0\}$. 
Let $\{u_n\}\subset BV(\mathbb{R}^N)\setminus\{0\}$ be a sequence with $u_n\to u\neq0$ in $BV(\mathbb{R}^N)$. We prove that $\tau(u_n)\to\tau(u)$ as $n\to\infty$.  

Firstly, we prove that $\{\tau(u_n)\}$ is bounded in $\mathbb{R}$. 
Indeed, assume for contradiction that there exists a subsequence (still denoted by itself) such that $\tau(u_n)\to+\infty$ as $n\to\infty$. Then by Lemma \ref{lem:2-7}, (\ref{eq:I(tau-u)}) and (\ref{ineq:f-F-1}), we have
  \begin{eqnarray}\label{lem:2-6-2}
    0 &\leq& e^{-\tau(u_n)}I(\tau(u_n)*u_n)\nonumber\\
    &=&\int_{\mathbb{R}^N}|Du_n|-e^{-(N+1)\tau(u_n)}\int_{\mathbb{R}^N}F(e^{N\tau(u_n)}u_n)dx
    -\frac{1}{1^*}e^{(1^*-1)\tau(u_n)}\int_{\mathbb{R}^N}|u_n|^{1^*}dx\nonumber\\
    &\leq&\int_{\mathbb{R}^N}|Du_n|-\frac{1}{1^*}e^{(1^*-1)\tau(u_n)}\int_{\mathbb{R}^N}|u_n|^{1^*}dx\nonumber\\
    &\to-\infty&.
  \end{eqnarray}
  This contradiction implies that $\{\tau(u_n)\}$ is bounded from above. Now assume that up to a subsequence, $\tau(u_n)\to-\infty$ as $n\to\infty$. 
  Note that $\tau(u)*u_n\to\tau(u)*u$ in $BV(\mathbb{R}^N)$. By Lemma \ref{lem:2-6}, we have
  \begin{eqnarray*}
    0 &<& I(\tau(u)*u)=\lim_{n\to\infty}I(\tau(u)*u_n)\nonumber\\
    &\leq&\liminf_{n\to\infty}I(\tau(u_n)*u_n)\nonumber\\
    &\leq&\liminf_{n\to\infty}\int_{\mathbb{R}^N}|D(\tau(u_n)*u_n)|\nonumber\\
    &=&\lim_{n\to\infty}e^{\tau(u_n)}\int_{\mathbb{R}^N}|Du_n|=0,
  \end{eqnarray*}
  a contradiction. This shows that $\{\tau(u_n)\}$ is bounded in $\mathbb{R}$. Therefore, we may assume that $\tau(u_n)\to\tau_0$ as $n\to\infty$, and this together with $u_n\to u$ in $BV(\mathbb{R}^N)$ yields $\int_{\mathbb{R}^N}|D(\tau(u_n)*u_n)|\to\int_{\mathbb{R}^N}|D(\tau_0*u)|$ as $n\to\infty$. 
  
  Secondly, we prove that $P(\tau_0*u)=0$. This is because
  \begin{eqnarray*}
    P(\tau_0*u) &=& \int_{\mathbb{R}^N}|D(\tau_0*u)|-N\int_{\mathbb{R}^N}\widetilde{F}(\tau_0*u)dx
    -\int_{\mathbb{R}^N}|\tau_0*u|^{1^*}dx\\
    &=&\lim_{n\to\infty}\left\{\int_{\mathbb{R}^N}|D(\tau(u_n)*u_n)|-N\int_{\mathbb{R}^N}\widetilde{F}(\tau(u_n)*u_n)dx
    -\int_{\mathbb{R}^N}|\tau(u_n*u_n)|^{1^*}dx\right\}\\
    &=&\lim_{n\to\infty}P(\tau(u_n)*u_n)=0.
  \end{eqnarray*}
  
  By the uniqueness of $\tau(u)$, we conclude that $\tau(u)=\tau_0$. This proves that $\tau(u_n)=\tau(u)$ as $n\to\infty$.  
\end{proof}

\begin{lem}\label{lem:2-9}
  Suppose that ($f_1$)-($f_4$) hold. 
  Then the mapping $m\mapsto E_m$ is continuous for $m>0$.
\end{lem}
\begin{proof}
  We prove that for any $m_k\to m$ as $k\to\infty$, we have $E_{m_k}\to E_m$. We divide the proof into two steps.
  
  Firstly, we prove that $\limsup_{k\to\infty}E_{m_k}\leq E_m$.
  
  Let $u\in\mathcal{P}_m$ and set $u_k=\frac{m_k}{m}u$. 
  Then $u_k\in S_{m_k}$ and $u_k\to u$ in $BV(\mathbb{R}^N)$, $\tau(u_k)\to\tau(u)=0$. Thus by the definition of $E_{m_k}$, we have
  \[
  E_{m_k}\leq I(\tau(u_k)*u_k)\leq I(u)+o_k(1).
  \]
  Then
  \[
  \limsup_{k\to\infty}E_{m_k}\leq I(u).
  \]
  Taking the infimum over $\mathcal{P}_m$ with respect to $u$, we have
  \[
  \limsup_{k\to\infty}E_{m_k}\leq E_m.
  \] 
  
  Secondly, we prove that $E_m\leq\liminf_{k\to\infty}E_{m_k}$.
  
  For $k\in\mathbb{N}^+$, there exists $v_k\in\mathcal{P}_{m_k}$ such that 
  \[
  I(v_k)<E_{m_k}+\frac{1}{k}.
  \]
  Let $t_k=(\frac{m}{m_k})^{\frac{1}{N}}$ and define $\widetilde{v}_k(\cdot)=v_k(\frac{\cdot}{t_k})$. Then $t_k\to1$ as $k\to\infty$ and $\widetilde{v}_k\in S_m$.
  
  Claim (1): $\{v_k\}$, $\{\widetilde{v}_k\}$ are bounded sequences in $BV(\mathbb{R}^N)$. In fact, since $\limsup_{k\to\infty}E_{m_k}\leq E_m$, $E_{m_k}$ is bounded, so we have 
  \[
  I(v_k)<E_m+1.
  \]
  This together with $\|v_k\|_1\leq m+1$ and Lemma \ref{lem:2-7} implies that $\{v_k\}$ is bounded in $BV(\mathbb{R}^N)$, and so is $\{\widetilde{v}_k\}$.
  
  Claim (2): $\limsup_{k\to\infty}\int_{\mathbb{R}^N}|\widetilde{v}_k|^{1^*}dx>0$. In fact, assume for contradiction that $\widetilde{v}_k\to0$ in $L^{1^*}(\mathbb{R}^N)$. Since $t_k$ is bounded, we have $v_k\to0$ in $L^{1^*}(\mathbb{R}^N)$. 
  Then by the interpolation inequality and the fact that $\|v_k\|_1\leq m+1$, we have $\int_{\mathbb{R}^N}\widetilde{F}(v_k)dx\to0$, and thus
  \[
  \int_{\mathbb{R}^N}|Dv_k|=N\int_{\mathbb{R}^N}\widetilde{F}(v_k)dx+\int_{\mathbb{R}^N}|v_k|^{1^*}dx\to0.
  \]
  Since $v_k\in\mathcal{P}_{m_k}$, we have $P(v_k)=0$. 
  But according to \eqref{cor:2-4}, the above limit implies that $P(v_k)\geq\frac{1}{2}\|Dv_k\|_1>0$, a contradiction.
  
  Claim (3): $\limsup_{k\to\infty}\tau(\widetilde{v}_k)<+\infty$. In fact, suppose that up to a subsequence, $\tau(\widetilde{v}_k)\to+\infty$ as $k\to\infty$. 
  Note that $P(\tau(\widetilde{v}_k)*\widetilde{v}_k)=0$, $F(\tau(\widetilde{v}_k)*\widetilde{v}_k)>0$, and $\{\widetilde{v}_k\}$ is bounded in $BV(\mathbb{R}^N)$.
  By Lemma \ref{lem:2-7} and \eqref{eq:I(tau-u)}, we can argue as in (\ref{lem:2-6-2}) to obtain
  \begin{eqnarray*}
    0 \leq e^{-\tau(\widetilde{v}_k)}I(\tau(\widetilde{v}_k)*\widetilde{v}_k) \to-\infty,
  \end{eqnarray*}
  which is a contradiction. Therefore, $\{\tau(\widetilde{v}_k)\}$ is bounded.
  
  Recall that $\widetilde{v}_k\in S_m$ and $v_k\in S_{m_k}$.
  By the definition of $E_m$ and Lemma \ref{lem:2-6}, we have
  \begin{eqnarray*}
    E_m &\leq& I(\tau(\widetilde{v}_k)*\widetilde{v}_k) \\
     &\leq& I(\tau(\widetilde{v}_k)*v_k)+|I(\tau(\widetilde{v}_k)*\widetilde{v}_k)-I(\tau(\widetilde{v}_k)*v_k)| \\
     &\leq& I(v_k)+|I(\tau(\widetilde{v}_k)*\widetilde{v}_k)-I(\tau(\widetilde{v}_k)*v_k)| \\
     &\leq& E_{m_k}+\frac{1}{k}+M(k),
  \end{eqnarray*}
  where 
  \begin{eqnarray*}
    M(k) &=& |I(\tau(\widetilde{v}_k)*\widetilde{v}_k)-I(\tau(\widetilde{v}_k)*v_k)| \\
     &=& |t_k^{N-1}-1|\int_{\mathbb{R}^N}|D(\tau(\widetilde{v}_k)*v_k)|
     -|t_k^N-1|\int_{\mathbb{R}^N}F(\tau(\widetilde{v}_k)*v_k)dx \\
     &&-\frac{1}{1^*}|t_k^N-1|\int_{\mathbb{R}^N}|\tau(\widetilde{v}_k)*v_k|^{1^*}dx.
  \end{eqnarray*}
  Since $\{t_k\}$ and $\{\tau(\widetilde{v}_k)\}$ are bounded and $\{v_k\}$ is bounded in $BV(\mathbb{R}^N)$, we derive that $\{M(k)\}$ is bounded and $M(k)\to0$ as $k\to\infty$. Therefore
  \[
  E_m\leq E_{m_k}+\frac{1}{k}+o_k(1),
  \]
  which yields $E_m\leq\liminf_{k\to\infty}E_{m_k}$.
  
  In summary, we conclude that
  $
  \lim_{k\to\infty}E_{m_k}=E_m.
  $
  Thus $E_m$ is continuous with respect to $m$.
\end{proof}

\begin{lem}\label{lem:2-10}
  Assume that ($f_1$)-($f_4$) hold. Then the function $m\mapsto E_m$ is nonincreasing on $(0,+\infty)$.
\end{lem}
\begin{proof}
  We prove that for any $m>m'>0$ and $\varepsilon>0$,
  \[
  E_m\leq E_{m'}+\varepsilon.
  \]
  In fact, by the definition of $E_{m'}$, there exists $u\in\mathcal{P}_{m'}$ such that
  \begin{eqnarray}\label{lem:2-10-1}
    I(u)\leq E_{m'}+\frac{\varepsilon}{2}.
  \end{eqnarray}
  Let $\chi\in C_0^\infty(\mathbb{R}^N)$ be radial and satisfy
  \begin{eqnarray*}
    \chi(x)=\left\{
              \begin{array}{ll}
                1, & \hbox{$|x|\leq1$;} \\
                \in[0,1], & \hbox{$|x|\in(1,2)$;} \\
                0, & \hbox{$|x|\geq2$.}
              \end{array}
            \right.
  \end{eqnarray*}
For sufficiently small $\delta>0$, we define
\[
u_\delta(x)=u(x)\chi(\delta x)\in BV(\mathbb{R}^N)\setminus\{0\}.
\]
Then $u_\delta\to u$ in $BV(\mathbb{R}^N)$ as $\delta\to0^+$, and by Lemma \ref{lem:2-8}, 
$\lim_{\delta\to0^+}\tau(u_\delta)=\tau(u)=0$. 
Thus
\[
\int_{\mathbb{R}^N}|D(\tau(u_\delta)*u_\delta)|\to\int_{\mathbb{R}^N}|D(\tau(u)*u)|=\int_{\mathbb{R}^N}|Du|.
\]
Fix $\delta>0$ small enough such that
\begin{eqnarray}\label{lem:2-10-2}
  I(\tau(u_\delta)*u_\delta)\leq I(u)+\frac{\varepsilon}{4}.
\end{eqnarray}
Take $v\in C_0^\infty(\mathbb{R}^N)$ such that $\mbox{supp}(v)\subset B_{1+{4}/{\delta}}(0)\setminus B_{{4}/{\delta}}(0)$, and define
\[
\widetilde{v}=\frac{m-\|u_\delta\|_1}{\|v\|_1}v.
\]
Note that $\mbox{supp}(u_\delta)\subset B_{2/\delta}(0)$. For $\kappa\leq0$, $\mbox{supp}(u_\delta)\cap\mbox{supp}(\kappa*\widetilde{v})=\emptyset$. Define $w_\kappa=u_\delta+\kappa*\widetilde{v}$. We have $w_\kappa\in S_m$.
We claim that $\tau(w_\kappa)$ is bounded from above as $\kappa\to-\infty$. 
Indeed, by Lemma \ref{lem:2-7}, $I(\tau(w_\kappa)*w_\kappa)\geq0$; moreover, $w_\kappa\to u_\delta$ a.e. in $\mathbb{R}^N$ as $\kappa\to-\infty$, so we can argue as (\ref{lem:2-6-2}) to get a contradiction that 
\[
0\leq e^{-\tau(w_\kappa)}I(\tau(w_\kappa)*w_\kappa)\to-\infty.
\]
This proves that the claim holds. Since $\tau(w_\kappa)+\kappa\to-\infty$ as $\kappa\to-\infty$, we have
\begin{eqnarray}\label{lem:2-10-21}
\|D[(\tau(w_\kappa)+\kappa)*\widetilde{v}]\|_1\to0,
\end{eqnarray}
\begin{eqnarray}\label{lem:2-10-22}
(\tau(w_\kappa)+\kappa)*\widetilde{v}\to0\quad\mbox{in~}L^{1+\frac{1}{N}}(\mathbb{R}^N).
\end{eqnarray}
Moreover, note that \eqref{lem:2-10-21} implies that $(\tau(w_\kappa)+\kappa)*\widetilde{v}$ is bounded in $BV(\mathbb{R}^N)$. 
Then, using \eqref{lem:2-10-22} and the H\"{o}lder inequality, we can prove that
\[
\|F((\tau(w_\kappa)+\kappa)*\widetilde{v})\|_1\to0.
\]
Therefore, by the embedding inequality, we obtain that for sufficiently small $\kappa<0$, 
\begin{eqnarray}\label{lem:2-10-3}
  I((\tau(w_\kappa)+\kappa)*\widetilde{v})\leq\frac{\varepsilon}{4}.
\end{eqnarray}
Thus by Lemma \ref{lem:2-6} and \eqref{lem:2-10-1}, \eqref{lem:2-10-2}, \eqref{lem:2-10-3}, we have
\begin{eqnarray*}
  E_m &\leq& I(\tau(w_\kappa)*w_\kappa) \\
   &=& I(\tau(w_\kappa)*u_\delta)+I(\tau(w_\kappa)*(\kappa*\widetilde{v})) \\
   &\leq& I(\tau(u_\delta)*u_\delta)+I((\tau(w_\kappa)+\kappa)*\widetilde{v}) \\
   &=& I(u)+\frac{\varepsilon}{4}+\frac{\varepsilon}{4} \\
   &=& E_{m'}+\varepsilon.
\end{eqnarray*}
This completes the proof.
\end{proof}

\begin{lem}\label{lem:2-13}
  Assume that ($f_1$)-($f_4$) and ($f_6$) hold. Then $E_m\to0^+$ as $m\to+\infty$.
\end{lem}
\begin{proof}
  Fix $u\in S_1\cap L^\infty(\mathbb{R}^N)$. For any $m>0$, set $u_m=mu\in S_m$. By Lemma \ref{lem:2-6}, there exists a unique $\tau_m\in\mathbb{R}$ such that $\tau_m*u_m\in\mathcal{P}_m$. Then by Lemma \ref{lem:2-7}, we have
  \[
  0<E_m\leq I(\tau_m*u_m)\leq e^{\tau_m}\int_{\mathbb{R}^N}|Du_m|=me^{\tau_m}\int_{\mathbb{R}^N}|Du|.
  \]
  We prove that $me^{\tau_m}\to0^+$ as $m\to+\infty$, which yields the conclusion of the lemma.
  
  {Claim (1):} $me^{N\tau_m}\to0^+$ as $m\to+\infty$. 
  In fact, since $\tau_m*u_m\in\mathcal{P}_m$, we have
  \[
  \int_{\mathbb{R}^N}|D(\tau_m*u_m)|=N\int_{\mathbb{R}^N}\widetilde{F}(\tau_m*u_m)dx+\int_{\mathbb{R}^N}|\tau_m*u_m|^{1^*}dx.
  \]
  That is,
  \[
  me^{\tau_m}\int_{\mathbb{R}^N}|Du|=Ne^{-N\tau_m}\int_{\mathbb{R}^N}\widetilde{F}(e^{N\tau_m}mu)dx
  +m^{1^*}e^{1^*\tau_m}\int_{\mathbb{R}^N}|u|^{1^*}dx.
  \]
  This yields
  \begin{eqnarray}\label{lem:2-13-1}
    \int_{\mathbb{R}^N}|Du| &=& Nm^{-1}e^{-(N+1)\tau_m}\int_{\mathbb{R}^N}\widetilde{F}(e^{N\tau_m}mu)dx
    +m^{1^*-1}e^{(1^*-1)\tau_m}\int_{\mathbb{R}^N}|u|^{1^*}dx.
  \end{eqnarray}
Note that $\widetilde{F}(t)\geq0$ by \eqref{ineq:f-F-1}.
From the above inequality we have
  \begin{eqnarray*}
    \int_{\mathbb{R}^N}|Du| &\geq& m^{1/N}(me^{N\tau_m})^{1/N(N-1)}\int_{\mathbb{R}^N}|u|^{1^*}dx.
  \end{eqnarray*}
  This implies that
  \[
  \lim_{m\to+\infty}me^{N\tau_m}=0.
  \]
  {Claim (2):} {$me^{\tau_m}\to0^+$ as $m\to+\infty$.} 
Assumption ($f_6$) is used here. 
In fact, by \eqref{ineq:f-F-1} and ($f_6$), for any $\varepsilon>0$, there exists sufficiently small $\delta>0$ such that for any $|t|\leq\delta$, we have
  \begin{eqnarray}\label{lem:2-13-2}
\widetilde{F}(t)\geq\frac{1}{N}F(t)\geq\frac{1}{\varepsilon}|t|^{1^*}.
  \end{eqnarray}
Note that $u\in L^\infty(\mathbb{R}^N)$.
By Claim (1), $me^{N\tau_m}u\to0$ a.e. in $\mathbb{R}^N$ as $m\to+\infty$. Thus from (\ref{lem:2-13-1}) and (\ref{lem:2-13-2}) we get
  \begin{eqnarray*}\label{lem:2-13-3}
    \int_{\mathbb{R}^N}|Du| &\geq& Nm^{-1}e^{-(N+1)\tau_m}\frac{1}{\varepsilon}\int_{\mathbb{R}^N}|e^{N\tau_m}mu|^{1^*}dx\\
&=&\frac{1}{\varepsilon}N(me^{\tau_m})^{1/(N-1)}\int_{\mathbb{R}^N}|u|^{1^*}dx.
  \end{eqnarray*}
By the arbitrariness of $\varepsilon$, we conclude that Claim (2) holds.
  This completes the proof.
\end{proof}

\begin{lem}\label{lem:2-14}
  Assume that ($f_1$)-($f_4$) hold. Then $0<\gamma_m\leq E_m\to0^+$.
\end{lem}
\begin{proof}
For arbitrary $\varepsilon>0$, we construct a path in $\Gamma_m$ (defined in (\ref{lem:4-2-5})) which linking $u_1$ and $u_2$ and having maximal energy less than $E_m+\varepsilon$.

(1) By the definition of $E_m$, $\exists u\in\mathcal{P}_m$ such that $I(u)\leq E_m+\frac{\varepsilon}{2}$.
Let $u_\delta$ be defined as in the proof of Lemma \ref{lem:2-10}, we define $\overline{u}_\delta=\frac{m}{\|u_\delta\|_1}u_\delta$, then $\overline{u}_\delta\in S_m$. By Lemma \ref{lem:2-6},
there exists a unique number $\tau(\overline{u}_\delta)\in\mathbb{R}$ such that $P(\tau(\overline{u}_\delta)*\overline{u}_\delta)=0$ and $I(\tau(\overline{u}_\delta)*\overline{u}_\delta)>I(\tau*\overline{u}_\delta)$ for any $\tau\neq\tau(u)$. 
Then arguing like for (\ref{lem:2-10-2}), we have for $\delta>0$ small enough,
\begin{eqnarray*}
  E_m\leq I(\tau(\overline{u}_\delta)*\overline{u}_\delta) \leq I(u)+\frac{\varepsilon}{2}\leq E_m+{\varepsilon}.
\end{eqnarray*}

(2) By Lemma \ref{lem:2-5}, $\exists \tau_1<0$ such that
\begin{eqnarray*}
  I(\tau_1*\overline{u}_\delta) &=& e^{\tau_1}J_1(\overline{u}_\delta)\leq\frac{1}{4}E_m.
\end{eqnarray*}
Denote $e_1=(1,0,\cdots,0)\in\mathbb{R}^N$ and define $\overline{u}_{\delta,e_1}(x)=\overline{u}_\delta\left(x-\frac{4}{e^{\tau_1}\delta}e_1\right)$.
Since $I$ is translational invariant, we have
\begin{eqnarray*}
  E_m\leq I(\tau(\overline{u}_{\delta,e_1})*\overline{u}_{\delta,e_1}) \leq E_m+\varepsilon
\end{eqnarray*}
and
\begin{eqnarray*}
  I(\tau_1*\overline{u}_{\delta,e_1}) = I(\tau_1\overline{u}_{\delta})\leq\frac{1}{4}E_m.
\end{eqnarray*}
Now we define $u_{1,\delta}(x)=u_1(x)\chi(\delta x)$ and $\overline{u}_{1,\delta}=\frac{m}{\|u_{1,\delta}\|_1}u_{1,\delta}$.
Then $\overline{u}_{1,\delta}\in S_m$ and 
\begin{eqnarray*}
  \mbox{supp}(\tau_1*\overline{u}_{\delta,e_1})\cap\mbox{supp}(\overline{u}_{1,\delta})=\emptyset.
\end{eqnarray*}
Without loss of generality, we assume $I(u_1)\leq\frac{E_m}{4}$. Then for $\delta>0$ small enough, $I(\overline{u}_{1,\delta})\leq\frac{E_m}{2}$ and $\forall t\in[0,1]$,
\begin{eqnarray*}
  I\left((1-t)\overline{u}_{1,\delta}+t(\tau_1*\overline{u}_{\delta,e_1})\right)
  \leq(1-t)I(\overline{u}_{1,\delta})+tI(\tau_1*\overline{u}_{\delta,e_1})
  \leq\frac{3}{4}E_m.
\end{eqnarray*}
Let $u_{1,t}=tu_1+(1-t)\overline{u}_{1,\delta}$ and $\overline{u}_{1,t}=\frac{m}{\|u_{1,t}\|_1}u_{1,t}$.
Then $\overline{u}_{1,t}\in S_m$, $\|u_{1,t}\|_1=m+o(\delta)$ and
\begin{eqnarray*}
  I(\overline{u}_{1,t}) &\leq& \frac{m t}{\|u_{1,t}\|_1}J_1(u_1) + \frac{m(1-t)}{\|u_{1,t}\|_1}J_1(\overline{u}_{1,\delta})\\
  &\leq& \frac{m t}{\|u_{1,t}\|_1}\frac{1}{4}E_m+\frac{m(1-t)}{\|u_{1,t}\|_1}\frac{1}{2}E_m\\
  &\leq&\frac{1}{2}E_m+o(\delta).
\end{eqnarray*}

(3) By Lemma \ref{lem:2-5} again, $\exists \tau_2>0$ large enough such that
\begin{eqnarray*}
  I(\tau_2*\overline{u}_{\delta,e_1}) &=& I({\tau_2}*\overline{u}_\delta)\leq\frac{1}{4}E_m.
\end{eqnarray*}
We define $u_{2,\delta}(x)=u_2(x)\chi(\delta x)$ and $\overline{u}_{2,\delta}=\frac{m}{\|u_{2,\delta}\|_1}u_{2,\delta}$.
Then $\overline{u}_{2,\delta}\in S_m$ and 
\begin{eqnarray*}
  \mbox{supp}(\tau_2*\overline{u}_{\delta,e_1})\cap\mbox{supp}(\overline{u}_{2,\delta})=\emptyset. 
\end{eqnarray*}
We may assume $I(u_2)\leq\frac{E_m}{4}$. Then for $\delta>0$ small enough, $I(\overline{u}_{2,\delta})\leq\frac{E_m}{2}$ and $\forall t\in[0,1]$,
\begin{eqnarray*}
  I\left(t\overline{u}_{2,\delta}+(1-t)(\tau_2*\overline{u}_{\delta,e_1})\right) \leq tI(\overline{u}_{2,\delta})+(1-t)I(\tau_2*\overline{u}_{\delta,e_1})\leq\frac{3}{4}E_m.
\end{eqnarray*}
Let $u_{2,t}=t\overline{u}_{2,\delta}+(1-t)u_2$ and $\overline{u}_{2,t}=\frac{m}{\|u_{2,t}\|_1} u_{2,t}$.
Then $\overline{u}_{2,t}\in S_m$, $\|u_{2,t}\|_1=m+o(\delta)$ and
\begin{eqnarray*}
  I(\overline{u}_{2,t}) &=& I\left(\frac{mt}{\|u_{2,t}\|_1}\overline{u}_{2,t}+\frac{m(1-t)}{\|u_{2,t}\|_1}u_2\right)\\
  &\leq&\frac{m}{\|u_{2,t}\|_1}I(u_2)+o(\delta)\\
  &\leq&\frac{1}{2}E_m.
\end{eqnarray*}

In conclusion, we have construct a path linking $u_1\to\overline{u}_{1,\delta}\to\tau_1*\overline{u}_{\delta,e_1}\to \tau(\overline{u}_{\delta,e_1})*\overline{u}_{\delta,e_1}\to\tau_2*\overline{u}_{\delta,e_1}\to\overline{u}_{2,\delta}\to u_2$ with maximal energy less than $E_m+\varepsilon$. By the arbitrariness of $\varepsilon$, we derive $\gamma_m\leq E_m$. 
This completes the proof.
\end{proof}

\begin{proof}[\it Proof of Theorem \ref{thm:m1}] 
By Proposition \ref{prop:4-5}, there is a bounded minimizing sequence $\{v_n\}\subset\ S_m$ for $I|_{S_m}$ at level $\gamma_m$, satisfying $P(v_n)\to0$. 
Furthermore, Proposition \ref{prop:4-8} shows that this sequence is compact in $L^r(\mathbb{R}^N)$, $\forall r\in[1,1^*]$ provided $\gamma_m\in(0,\frac{1}{N}S^N)$ is small enough. 
Lemma \ref{lem:2-13} and Lemma \ref{lem:2-14} imply that this is true because $0<\gamma_n\leq E_m\to0^+$ as $m\to+\infty$. This means that we can choose $m$ sufficiently large to recover the compactness of $\{v_n\}$. 
Then by Proposition \ref{prop:4-9}, there exist  $v_0\in S_m$ with $P(v_0)=0$, $I(v_0)=\gamma_m$, and $\lambda>0$ such that \eqref{prop:4-9-1} holds. 
This proves that $(\lambda,v_0)$ is a solution pair to problem \eqref{eq:1-Lap}-\eqref{eq:constraint}.
\end{proof}

At the end of this paper, we give some remarks.

\begin{rem}
Let $v_0$ be the solution obtained in Theorem \ref{thm:m1} or \ref{thm:m3}. Then
\\
(1) $v_0\in S_m$;~\\
(2) $P(v_0)=0$;~\\
(3) $\gamma_m=E_m=I(v_0)$.
\\
Indeed, on the one hand we have $\gamma_m\leq E_m$ according to Lemma \ref{lem:2-14};
  on the other hand, in the proof of Proposition \ref{prop:4-9}, we find a minimizing sequence in $S_m$ that converges strongly in $L^1(\mathbb{R}^N)$ norm to $v_0$ in $S_m$ with $P(v_0)=0$ and $I(v_0)=\gamma_m$. This means that, in fact, $\gamma_m=E_m$.
\end{rem}

\begin{rem}
  The method for constructing a minimizing Palais-Smale sequence under constraints can be extended to the general case, more precisely, to a general positive convex Lipschitz continuous functional plus a $C^1$ functional.
\end{rem}

\begin{rem}
  It is interesting to remove the assumption that "$m>0$ is sufficiently large" for the special problem $f(u)=|u|^{q-2}u$, $q\in(1+\frac{1}{N},1^*)$, in Theorem \ref{thm:m3}.
\end{rem}


\subsection*{Conflict of Interest}
The authors declare no conflict of interest.

\subsection*{Data availability}
No data was used for the research described in the article.

\subsection*{Declaration of generative AI and AI-assisted technologies in the manuscript preparation process}
During the preparation of this work the authors used AI-assisted technologies in order to improve the English language and readability. 
After using the tools, the authors reviewed and edited the content as needed and take full responsibility for the content of the published article.

\subsection*{CRediT authorship contribution statement}
Zhouxin Li: Conceptualization, Methodology, Writing – original draft. \\
Ruishu Liu: Validation, Writing – review and editing.



\begin{thebibliography}{99}

\bigbreak

\small
\bibitem{AlPi17} Alves C. O., Pimenta M.T.O., On existence and concentration of solutions to a class of quasilinear problems involving the 1-Laplace operator, Calc. Var. Partial Differential Equations, {\bf 56} (2017), paper No.143, 24 pp.

\bibitem{ABC97} Ambrosetti, A., Badiale, M., Cingolani, S., Semiclassical states of nonlinear Schr\"{o}dinger equations, Arch. Ration. Mech. Anal., {\bf 140} (1997), 285-300.
    
\bibitem{AFP00} Ambrosio L., Fusco N., Pallara D., Functions of bounded variation and free discontinuous problems,
Clarendon Press, Oxford Press, 2000.

    
\bibitem{Anze85} Anzellotti G., The Euler equation for functionals with linear growth, Trans. Am. Math. Soc., {\bf 290}(2) (1985), 483-501.

\bibitem{ApSe21} Applolloni L, Secchi S., Normalized solutions for the fractional NLS with mass supercritical nonlinearity, J. Differential Equations, {\bf 286} 2021, 248-283.

\bibitem{ABM06} Attouch, H., Buttazzo, G., Michaille, G., Variational Analysis in Sobolev and BV Spaces: Applications to PDEs and Optimization, MPS-SIAM, Philadelphia, 2006.

\bibitem{BaVa13} Bartsch T, de Valeriola S., Normalized solutions of nonlinear Schr\"{o}dinger equations, Arch. Math., {\bf 100} (2013), 75-83.
    
\bibitem{BWW05} Bartsch T., Wang Z.-Q. and Willem M., The Dirichlet problem for superlinear elliptic equa-tions, In: Stationary Partial Differential Equations. Handb. Differ. Equ., vol. II, pp. 1-55. Elsevier/North-Holland, Amsterdam, 2005.
    
    
\bibitem{Chang09} Chang, K., The spectrum of the 1-Laplace operator, Commun. Contemp. Math., {\bf 9}(4) (2009), 515-543.

\bibitem{ChPi21} J. C. O Chata, M. T. O. Pimenta, A Berestycki-Lions'  type result to a quasilinear elliptic problem involving the 1-Laplacian operator, J. Math. Anal. Appl., {\bf 500} (2021), Paper No.125074, 21pp.

\bibitem{Clar76} F. H. Clarke, A new approach to Lagrange multipliers, Math. Oper. Res., {\bf 1} (1976), 165-174.

\bibitem{CoDM93} Corvellec J.N., Degiovanni M., Marzocchi M., Deformation properties for continuous functionals and critical point theory, Topol. Methods Nonlinear Anal., {\bf 1} (1993), 151-171.

\bibitem{DeSc98} Degiovanni M., Schuricht F., Buckling of nonlinearly elastic rods in the presence of obstacles treated by nonsmooth critical point theory, Math. Ann., {\bf 311} (1998), 675-728.
    
\bibitem{DeMa94} Degiovanni M., Marzocchi M., A critical point theory for nonsmooth functional, Ann. Mat. Pura Appl., {\bf 167} (4) (1994), 73-100.

\bibitem{DeMa09} Degiovanni, M., Magrone, P., Linking solutions for quasilinear equations at critical growth involving the 1-Laplace operator, Calc. Var. Partial Differential Equations, {\bf 36} (2009), 591-609.

\bibitem{PiFe96} del Pino, M., Felmer, P.L., Local mountain pass for semilinear elliptic problems in unbounded domains, Calc. Var. Partial Differential Equations, {\bf 4}(2) (1996), 121-137.

\bibitem{FiPi18} Figueiredo G.M., Pimenta M.T.O., Existence of bounded variation solutions for a 1-Laplace problem with vanishing potentials, J. Math. Anal. Appl., {\bf 459} (2018), 861-878.
    
\bibitem{FlWe86} Floer, A., Weinstein, A., Nonspreading wave packets for the cubic Schr\"{o}dinger equation with a bounded potential, J. Funct. Anal. {\bf 69}(3) (1986), 397-408.

\bibitem{Gui96} Gui, C., Existence of multi-bump solutions for nonlinear Schr\"{o}dinger equations via variational method, Comm. Partial Differ. Equ., {\bf 21}(5-6) (1996), 787-820.

\bibitem{IkTa19} Ikoma N, Tanaka K., A note on deformation argument for L2 normalized solutions of nonlinear Schr\"{o}dinger equations and systems, Adv Differ Equ, {\bf 24} (2019), 609-646.

\bibitem{Jean97} Jeanjean L., Existence of solutions with prescribed norm for semilinear elliptic equations, Nonlinear Anal., {\bf 28} 1997, 1633-1659.
    
\bibitem{JeLu20} Jeanjean  L., Lu S.-S., A mass supercritical problem revisited, Calc. Var. Partial Differential Equations, {\bf 59} (2020), Paper No. 174, 43 pp.
    
\bibitem{KaSc07} Kawohl B., Schuricht F., Dirichlet problems for the 1-Laplace operator, including the eigenvalue problem, Commun. Contemp. Math., {\bf 9}(4) (2007), 525-543.
    
\bibitem{LiL21} Li Z., Liu R., Existence and concentration behavior of solutions to 1-Laplace equations on $R^N$, J. Differential Equations, {\bf 272} (2021), 399-432.
    
\bibitem{LLJ12} Li Z., Liu S., Jia R., Nonnegative solutions to 1-Laplace equations at critical growth, Sci. Sin., Math., {\bf 42} (8) (2012), 775-785. (in Chinese)
    
\bibitem{LiZo24} Li Q, Zou W., Normalized ground states for Sobolev critical nonlinear Schr\"{o}dinger equation in the $L^2$-supercritical case, Discrete Contin Dyn Syst, {\bf 44}(1) 2024, 205-227.

    
\bibitem{Lions85} Lions P. L., The concentration-compactness principle in the calculus of variation, the limit case, parts 1, 2., Rev. Mat. Iberoamericana, {\bf 1} (1985), 145-201, 45-121.

\bibitem{MaWi89} Mawhin J., Willem M., Critical point theory and Hamiltonian system, vol.74, Appl. Math. Sci., Springer, Heidelberg, 1989.

\bibitem{Oh90} Oh, Y.G., On positive multi-lump bound states of nonlinear Schr\"{o}dinger equations under multiple well potential, Commun. Math. Phys., {\bf 131}(2) (1990), 223-253.

\bibitem{QiZo23} Qi S., Zou W., Mass threshold of the limit behavior of normalized solutions to Schr\"{o}dinger equations with combined nonlinearities, J. Differential Equations, {\bf 375} (2023), 172-205.

\bibitem{Rabi92} Rabinowitz, P.H., On a class of nonlinear Schr\"{o}dinger equations, Z. Angew. Math. Phys., {\bf 43}(2) (1992), 270-291.
    
\bibitem{SaLe18} Salas A. M., de Leon S. S., Elliptic equations involving the 1-Laplacian and a subcritical source term, Nonlinear Analysis, {\bf 168} (2018), 50-66.

\bibitem{Soa201} Soave N., Normalized ground states for the NLS equation with combined nonlinearities, J. Differential Equations, {\bf 269} (2020), 6941-6987.

\bibitem{Soa202} Soave N., Normalized ground states for the NLS equation with combined nonlinearities: The Sobolev critical case, J Funct. Anal., {\bf 279} (2020), Paper No. 108610, 43 pp.

\bibitem{Szul86} Szulkin A., Minimax principle for lower semicontinuous functions and applications to nonlinear boundary value problems, Ann. Inst. H. Poincar\'{e}, {\bf 3}(2) (1986), 77-109.

\bibitem{Wang93} Wang, X., On concentration of positive bound states of nonlinear Schr\"{o}dinger equations, Commun. Math. Phys., {\bf 153}(2) (1993), 229-244.

\bibitem{WeWu22} Wei J., Wu Y., Normalized solutions for Schr\"{o}dinger equations with critical Sobolev exponent and mixed nonlinearities, J Funct. Anal., {\bf 283} (2022), Paper No. 109574, 46 pp.

\bibitem{Will96} Willem M., Minimax Theorems, Boston, Birkh\"{a}user, 1996.
    
\end{thebibliography}
\end{document}